\documentclass[review,onefignum,onetabnum]{siamart190516}

\usepackage{graphicx}
\usepackage{mathptmx}
\usepackage{siunitx}
\usepackage{xfrac}
\usepackage{physics}
\usepackage{mathtools}
\usepackage{caption}
\usepackage{subcaption} 
\usepackage{amsmath}
\usepackage{amssymb}
\usepackage{booktabs}
\usepackage[sort&compress,numbers]{natbib}
\usepackage{xspace}
\newcommand{\code}[1]{\texttt{#1}}
\newcommand{\petsc}[0]{\code{PETSc}\xspace}
\newcommand{\petscfd}[0]{\code{PetscFD}\xspace}
\newcommand{\scomp}[0]{\emph{stencil composition}\xspace}

\headers{Stencil Composition}{A. Mishra, D. Salac, and M. G. Knepley}

\title{On the order of accuracy for finite difference approximations of partial differential equations using stencil composition}

\author{Abhishek Mishra \and David Salac \and Matthew G. Knepley}

\author{Abhishek Mishra\thanks{Institute for Computational and Data Sciences, University at Buffalo, Buffalo, NY, 14260, USA.}
\and David Salac\footnotemark[1] \thanks{Dept of Mechanical and Aerospace Engineering, University at Buffalo, Buffalo, NY, 14260, USA.}
\and Matthew G. Knepley\footnotemark[1] \thanks{Dept of Computer Science and Engineering, University at Buffalo, Buffalo, NY, 14260, USA.}}

\ifpdf
\hypersetup{
  pdftitle={On the order of accuracy for finite difference approximations of partial differential equations using stencil composition},
  pdfauthor={A. Mishra, D. Salac, and M. G. Knepley}
}
\fi

\begin{document}
\nolinenumbers
\maketitle

\begin{abstract}
  Stencil composition uses the idea of function composition, wherein two stencils with arbitrary orders of derivative are composed to obtain a stencil with a derivative order equal to sum of the orders of the composing stencils. In this paper, we show how stencil composition can be applied to form finite difference stencils in order to numerically solve partial differential equations (PDEs). We present various properties of stencil composition and investigate the relationship between the order of accuracy of the composed stencil and that of the composing stencils. We also present comparisons between the stability restrictions of composed higher-order PDEs to their compact versions and numerical experiments wherein we verify the order of accuracy by convergence tests. To demonstrate an application to PDEs, a boundary value problem involving the two-dimensional biharmonic equation is numerically solved using stencil composition and the order of accuracy is verified by performing a convergence test. The method is then applied to the Cahn-Hilliard phase-field model. In addition to sample results in 2D and 3D for this benchmark problem, the scalability, spectral properties, and sparsity is explored.
\end{abstract}





\section{Introduction}

Partial differential equations (PDEs) have a wide variety of applications, ranging from engineering~\cite{stone} to biology~\cite{olsen}, as well as in machine learning applications such as image processing~\cite{tian}. The first step in numerically solving any such PDE requires a discretization technique which replaces the continuous equation by a discrete algebraic equation~\cite{arad}. The discretization technique involves approximating the derivative terms in the PDE by a numerical method such as the finite difference~\cite{fd1, fd2, fd3}, finite element~\cite{fe1, fe2, fe3}, or finite volume method~\cite{fv1, fv2, tang}, calculated at discrete points, or in other words, the grid points.

In this work we focus on the finite difference method, where a linear PDE is discretized at a central point via a linear combination of neighboring points. This combination of neighboring points and their associated weights is called a \emph{stencil}. For PDEs, stencils at points are coupled to the stencils at neighboring points, which leads to a coupled set of linear equations if the function is unknown. In particular, we are interested in situations where the PDE itself is written as a repeated series of derivatives, such as the biharmonic equation which can be stated as two applications of the Laplace operator, $\Delta\Delta u$, or where the method chosen to solve a PDE results from the multiple applications of operators. An example of the latter is the Closest Point Method (CPM), which is a technique to solve surface differential equations on embedded surfaces whereby interpolation and derivative stencils are combined~\cite{cp1, cp2}. The CPM uses the fact that if the solution on an embedded surface, such as those described by the level-set method~\cite{ls1, ls2}, is extended into the embedding space such that it is constant in the direction normal to the surface, then standard Cartesian derivatives will correspond to surface derivatives when interpolated back down to the surface. This allows for linear systems to be created that allows for the solution of differential equations on arbitrary surfaces in a systematic manner.

In both cases (biharmonic or CPM) the end result can be written as a series of (typically) sparse matrix-matrix products. For many reasons (numerical stability, linear system solver speed, memory pre-allocation, etc) it is often advantageous to obtain a single matrix representing these types of systems. The naive method would be to use each individual stencil to create individual matrices in memory and perform many matrix-matrix products. For very large systems it is necessary to have some information on the sparsity pattern of the resulting matrix, which is difficult to obtain for arbitrary systems. It would be advantageous to have a stencil of the final system \emph{before} the matrix is formed. This leads to the concept of \emph{stencil composition}, whereby one stencil is \emph{composed} with another stencil. Focusing on the composition of derivative stencils, this allows (for example) two stencils with arbitrary derivative orders of $\bar{a}$ and $\bar{b}$ to be composed to obtain a stencil with a derivative order of $\bar{a}+\bar{b}$. This composed stencil can then be used to create the matrix required for solution of the PDE. In this work we explore the use of stencil composition of lower-order derivative stencils to form a single higher-order derivative stencil. In addition to demonstrating that resultant stencil does approximate the desired derivative, we will also demonstrate that if the order of accuracy of these two stencils are $p$ and $q$ then the order of accuracy for the resulting stencil will be $\min(p,q)$, which demonstrates that stencil composition will not degrade the expected order-of-accuracy. We also explore the stability properties of the resulting matrices and demonstrate that they are not adversely affected and in certain cases the composed matrices are more stable than their compact counterparts.

The remainder of this paper is organized as follows. In section~\ref{sec:FDS}, we formulate the shorthand notation for the finite difference stencil using the Taylor series, encapsulated as a vector. This notation is demonstrated by deriving the first-order and second-order derivative stencils using this vector. In section~\ref{sec:composition}, the concept of stencil composition is introduced. The associativity of stencil composition is shown, as is the order of accuracy and stability. Section \ref{sec:2dcomp} extends the concept of stencil composition and its order of accuracy to higher-dimensions. Sections~\ref{sec:1dex} and \ref{sec:2dex} present numerical examples using some arbitrary functions in one and two dimensions respectively, along with convergence studies. Section \ref{sec:biharmonic} demonstrates the PDE application of stencil composition by numerically solving a biharmonic boundary value problem, and verifying the order of accuracy by performing a convergence test while in Section~\ref{sec:cahn} a benchmark problem, the Cahn-Hilliard equation, is analyzed. Lastly, in section \ref{sec:conc}, we draw some conclusions and discuss possible usage and applications of stencil composition.

%
%

\section{Numerical Discretization}

%
%
\subsection{Finite Difference Stencils}
\label{sec:FDS}

%
%
%

This section outlines the notation used in the remainder of this work. Let $f(\vb{x})$ be a function defined over a lattice in $\mathbb{R}^d$. We define the \emph{target} point $\vb{x}_0$ as the location where we wish to evaluate the function $f$ or some derivative of the function. Generally the target point does not need to lie on the lattice, but it will for derivative approximations, which is the assumption here. We define a \emph{source} point $\vb{x}_i\neq\vb{x}_0$ and can estimate the value of the function at the source point via a Taylor series centered at the target point,
\begin{align}\label{taylor}
  f(\vb{x}_i) = f(\vb{x}_0) + (\vb{x_i} - \vb{x}_0)^T\left\{D f(\vb{x}_0)\right\} + \frac{1}{2!} (\vb{x}_i - \vb{x}_0)^T \left\{D^2 f(\vb{x}_0)\right\} (\vb{x}_i - \vb{x}_0) + \cdots
\end{align}
where $D f(\vb{x}_0)$ is the gradient of $f(x)$ and $D^2 f(\vb{x}_0)$ is the Hessian, both evaluated at $\vb{x}_0$.

Assuming that both target and source points lie on a regular lattice with spacing $h$, we can represent the difference using the integer vector $\vb{u}_i \in \mathbb{Z}^d$, so that
\begin{align}
  \vb{x}_i - \vb{x}_0 = h \vb{u}_i.
\end{align}
We can then rewrite the series expansion \eqref{taylor} using multi-index notation,
\begin{align}\label{taylor2}
  f(\vb{x}_i) = \sum_{|\alpha| \ge 0} h^{\alpha}\dfrac{\vb{u}^\alpha_i}{\alpha!} f^{(\alpha)}(\vb{x}_0),
\end{align}
where $\vb{u}^\alpha_i$ refers to component-wise powers.
Let us introduce an integer $\beta$, which represents the shift of derivatives in any given direction,
\begin{align}
  f^{(\beta)}(\vb{x}_i) = \sum_{|\alpha| \ge 0} h^{\alpha} \dfrac{\vb{u}^\alpha_i}{\alpha!} f^{(\alpha+\beta)}(\vb{x}_0).
\end{align}
Therefore, when $\beta = 0$ we get \eqref{taylor2} and the Taylor series expansion \eqref{taylor}. For $\beta>0$ this results in a Taylor series of the $\beta^{th}$-derivative of $f(\vb{x}_i)$.
We can now compactly write this infinite series using the following notation,
\begin{align}
  f^{(\beta)}(\vb{x}_i) \longrightarrow \left(\vb{S}_i, \ \beta \right),
\end{align}
where $\left(\vb{S}_i, \beta\right)$ denotes the Taylor Series expansion of a single source point centered at the target point, and the infinite vector $\vb{S}_i$ contains coefficients associated with each $h^\alpha\vb{s}_{i,\alpha} f^{(\alpha+\beta)}$:
\begin{align}
    (\{1,1,\frac{1}{2!},\frac{1}{3!},\ldots\},\beta=0)=f(\vb{x}_0)+h f'(\vb{x}_0)+\frac{1}{2!}h^2 f''(\vb{x}_0)+\frac{1}{3!}h^3 f^{(3)}(\vb{x}_0)+\cdots
\end{align}
while
\begin{align}\label{eq:Sexample}
    (\{0,1,1,\frac{1}{2!},\frac{1}{3!},\ldots\},\beta=0)=&hf'(\vb{x}_0)+h^2 f''(\vb{x}_0)+\frac{1}{2!}h^3 f^{(3)}(\vb{x}_0)\nonumber \\
    &\qquad+\frac{1}{3!}h^4 f^{(4)}(\vb{x}_0)+\cdots
\end{align}
for a one-dimensional system with $u=1$. It must be noted that, if we divide $\vb{S}_i$ by $h^p$, the coefficients in the infinite vector $\vb{S}_i$ move $p$ slots to the left, and the $\beta$ increases by $p$. For instance, dividing \eqref{eq:Sexample} by $h$ shifts the coefficients left one slot and increases $\beta$ by one,
\begin{align}
    (\{1,1,\frac{1}{2!},\frac{1}{3!},\ldots\},\beta=1)=f'(\vb{x}_0)+h f''(\vb{x}_0)+\frac{1}{2!}h^2 f^{(3)}(\vb{x}_0)+\frac{1}{3!}h^3 f^{(4)}(\vb{x}_0)+\cdots.
\end{align}
The sequences centered at a given target point constitute a vector space, and thus it is possible to take linear combinations of $n$-different source points.

\begin{definition}
    A finite difference stencil approximating the $p^{th}$-derivative of $f^{(\beta)}(\vb{x}_0)$ with associated scalar weights $a_i \propto h^{-p}$, can be expressed using the following notation
    \begin{align}
        f^{(\beta+p)}(\vb{x}_0)&\approx\sum_i a_i f^{(\beta)}(\vb{x}_i) = \sum_i a_i \left(\sum_{|\alpha| \ge 0} h^{\alpha} \dfrac{\vb{u}^\alpha_i}{\alpha!} f^{(\alpha+\beta)}(\vb{x}_0)\right) \nonumber\\
        &\longrightarrow \sum_i a_i \left(\vb{S}_i, \beta\right)= \left(\vb{T}, \beta+p\right)= (\vb{T}, \bar{\beta}),
    \end{align}
    where, $\bar{\beta}$ indicates the increase in $\beta$ when weights are applied and the overall derivative order approximated.
\end{definition}
This can be alternatively written as
\begin{align}\label{eq:CompactNotation}
    t_{\alpha} = \sum_{|\alpha| \ge 0} \sum_i a_i \dfrac{\vb{u}^\alpha_i}{\alpha!}\; \text{such that} \sum_{|\alpha| \ge 0} t_{\alpha} h^{\alpha} f^{(\alpha+\beta)}(\vb{x}_0)\longrightarrow (\vb{T}, \bar{\beta}),
\end{align}
where $t_{\alpha}$ denotes the coefficient associated with $h^\alpha$.
The remainder of this work assumes that whenever the notation $\vb{T}$ is used, stencil weights have already been applied. The symbol $\bar{\beta}$ may thus be suppressed and stencil may be expressed using $(\vb{T}, \beta)$.

\begin{definition}\label{def:StencilRules}
A finite difference stencil expressed using the notation $(\vb{T}, \beta+p)$ approximating a $(\beta+p)^{th}$-order derivative with an order of accuracy $q$, must satisfy the following:
\begin{itemize}
    \item all coefficients associated with derivatives of order less than $\beta+p$ in a weighted sum of the Taylor series must go to zero, and thus,
    \item the first element in $\vb{T}$ must be equal to one, which represents the coefficient of the $(\beta+p)^{th}$ derivative. Therefore,
    \item the value of $\beta+p$ indicates what derivative order the stencil approximates. Moreover,
    \item all coefficients associated with derivatives of order greater than $\beta+p$ and less than $\beta+p+q$ must be equal to zero, and
    \item the coefficient associated with order of derivative $\beta+p+q$ must be non-zero.
\end{itemize}
\end{definition}
It is important to note that the first non-zero value that follows the first element in $\vb{T}$ provides the coefficient associated with the order of accuracy, $q$, as demonstrated below.

\subsubsection{First Derivative Stencil}

We will begin with an example in one dimension. The simplest approximation we can make is to use an evaluation to estimate the value of the target point, $x_0$, itself,
\begin{align}\label{fa}
  f(x_0) \longrightarrow (\{1, 0, 0, \ldots\},\ \beta = 0) = (\vb{S}_0,\beta=0).
\end{align}
Next, we could move our source evaluation point $x$ forward one lattice spacing $x_1 = x_0 + h$,
\begin{align}\label{fx1}
  f(x_1) \longrightarrow (\{1, 1, \frac{1}{2!}, \frac{1}{3!}, \ldots\},\ \beta = 0) = (\vb{S}_1,\beta=0),
\end{align}
or one step backward $x_{-1} = x_0 - h$,
\begin{align}\label{fx-1}
  f(x_{-1}) \longrightarrow (\{1, -1, \frac{1}{2!}, \frac{-1}{3!}, \ldots\},\ \beta = 0) = (\vb{S}_{-1},\beta=0).
\end{align}
Each of these is an $\mathcal{O}(h)$ approximation to $f(x_0)$ as
\begin{align}
  f(x_0) - f(x_1)    &\longrightarrow (\{0, -1, \frac{-1}{2!}, \ldots\},\ \beta = 0), \\
  f(x_0) - f(x_{-1}) &\longrightarrow (\{0,  1, \frac{-1}{2!}, \ldots\},\ \beta = 0),
\end{align}
and the first non-vanishing term is proportional to $h f'(x_0)$.

Applying first derivative weights $\vb{a} = \left(\dfrac{1}{2h}, \dfrac{-1}{2h}\right)$ to $\vb{S} = (\vb{S}_1, \vb{S}_{-1})$ from \eqref{fx1} and \eqref{fx-1} we get
\begin{align}
  \sum_i a_i \left(\vb{S}_i, \beta\right) &= \frac{1}{2h}(\vb{S}_1,\beta=0)-\frac{1}{2h}(\vb{S}_{-1},\beta=0) \nonumber \\
    &=\frac{1}{2h}(\{1, 1, \frac{1}{2!}, \frac{1}{3!}, \ldots\},\ \beta = 0)-\frac{1}{2h}(\{1, -1, \frac{1}{2!}, \frac{-1}{3!}, \ldots\},\ \beta = 0) \nonumber \\
  &=\frac{1}{2h}(\{0, 2, 0, \frac{2}{3!}, 0, \ldots\},\ \beta = 0)=(\{1, 0, \frac{1}{6}, 0, \ldots\},\ \beta = 1)\nonumber\\
  &\longrightarrow (\vb{T},\beta=1).
\end{align}
Since the first element of $\vb{T}$ is one, $\beta = 1$ therefore indicates that the stencil is an approximation of $f'(x_0)$. As the first non-zero value following the first element in $\vb{T}$ in the vector is in the location corresponding to $h^2$ this approximation is of order $\mathcal{O}(h^2)$.

The numerical approximation can thus be written in the following manner,
\begin{align}\label{f'stencil}
  f'(x_0) &\longrightarrow \frac{f(x_1) - f(x_{-1})}{2h} + \mathcal{O}(h^2).
\end{align}

\subsubsection{Second Derivative Stencil}

The second derivative can be approximated in a similar manner. In this case the stencil must zero out the coefficients associated with $f(x_0)$ and $f'(x_0)$, and result in a coefficient of one associated with $f''(x_0)$. Applying the weights $\vb{a} = \left(\dfrac{1}{h^2}, \dfrac{-2}{h^2}, \dfrac{1}{h^2}\right)$ to $\vb{S} = (\vb{S}_1, \vb{S}_0, \vb{S}_{-1})$ we obtain
\begin{align}\label{eq:CenterSecondDerivative}
  \sum_i a_i \left(\vb{S}_i, \beta\right) &= \frac{1}{h^2}(\vb{S}_1,\beta=0)-\frac{2}{h^2}(\vb{S}_{0},\beta=0)+\frac{1}{h^2}(\vb{S}_{-1},\beta=0)\nonumber \\
     &=\frac{1}{h^2}(\{1, 1, \frac{1}{2!}, \frac{1}{3!}, \ldots\},\ \beta = 0)-\frac{2}{h^2}(\{1, 0, 0, 0, \ldots\},\ \beta = 0) \nonumber \\
     & \qquad+\frac{1}{h^2}(\{1, -1, \frac{1}{2!}, \frac{-1}{3!}, \ldots\},\ \beta = 0)\nonumber \\
     &=\frac{1}{h^2}(\{0,0,1,0,\frac{1}{12},\ldots\},\beta=0)=(\{1,0,\frac{1}{12},\ldots\},\beta=2)\nonumber \\
     &\longrightarrow(\vb{T},\beta=2).
\end{align}
Clearly this is now an $\mathcal{O}(h^2)$ approximation to $f''(x_0)$.

%
%

\subsection{Stencil Composition}
\label{sec:composition}

We now introduce the concept of stencil composition, which makes use of the idea of function compositions. Just like a function composition, stencil composition is an operation which takes two stencils $A$ and $B$, with derivative orders of $\bar{a}$ and $\bar{b}$ respectively, and generates a stencil $C$ such that $C=B(A)$ with a derivative order of $\bar{a}+\bar{b}$. In this operation, the outer stencil $B$ is applied to the result obtained by applying the inner stencil $A$ to any function $f$. To formally introduce this concept let the two stencils be given as
\begin{align}
    A&=\sum_i a_i f^{(\beta)}(\vb{x}_i)=\sum_i a_i \left(\sum_{|\alpha| \ge 0} h^{\alpha} \dfrac{\vb{u}^\alpha_i}{\alpha!} f^{(\alpha+\beta)}(\vb{x}_0)\right) \label{eq:StencilA}\\
    &\textnormal{and}\nonumber\\
    B&=\sum_j b_j g^{(\gamma)}(\vb{x}_j)=\sum_j b_j \left(\sum_{|\alpha| \ge 0} h^{\alpha} \dfrac{\vb{v}^\alpha_j}{\alpha!} g^{(\alpha+\gamma)}(\vb{x}_0)\right)\label{eq:StencilB},
\end{align}
where the source points, $\vb{x}=\vb{x}_0+h\vb{u}$ and $\vb{x}=\vb{x}_0+h\vb{v}$, and the associated weights, $\vb{a}\propto h^{-\bar{a}}$ and $\vb{b}\propto h^{-\bar{b}}$, could differ between the two stencils.
For composition, the outer stencil, $B$ in this case, is written as working on function values, not derivatives, and therefore $\gamma=0$. The composition can then be written as
\begin{align}\label{eq:C=B(A)}
    C=B(A) = \sum_j b_j \sum_i a_i f^{(\beta)}(\vb{x}_i+\vb{x}_j) &= \sum_j \sum_i a_i b_j \left(\sum_{|\alpha| \ge 0} h^{\alpha} \dfrac{\left(\vb{u}_i+\vb{v}_j\right)^{\alpha}}{\alpha!} f^{(\alpha+\beta)}(\vb{x}_0)\right)\nonumber\\
    &\longrightarrow\sum_j\sum_i a_i b_j (\vb{S}_{i+j},\beta) = (\vb{T},\bar{\beta}=\beta+\bar{a}+\bar{b}).
\end{align}

As an example, we will derive a second derivative stencil using the composition of two first derivative stencils. As a reminder, the first derivative stencil \eqref{f'stencil} looks like
\begin{equation*}
  f'(x_0) \longrightarrow \frac{f(x_1) - f(x_{-1})}{2h} + \mathcal{O}(h^2),
\end{equation*}
and therefore $\vb{u}=\{1,-1\}=\vb{v}$ is the associated integer vector and $\vb{a}=\{1/(2h),-1/(2h)\}=\vb{b}$ are the stencil weights. Note that as we will be composing a stencil with itself, the integer vectors/weights of both the inner and outer stencil will be the same and $\beta=0$. The composition is then
\begin{align}\label{eq:FullComposition}
    B(A) &= \sum_j \sum_i a_i b_j  \left(\sum_{|\alpha| \ge 0} h^{\alpha} \dfrac{\left(\vb{u}_i+\vb{v}_j\right)^{\alpha}}{\alpha!} f^{(\alpha)}(\vb{x}_0)\right)\nonumber\\
        &=\sum_j b_j
            \Bigg[
                \dfrac{1}{2h}\left(\sum_{|\alpha| \ge 0} h^{\alpha} \dfrac{\left(1+\vb{v}_j\right)^{\alpha}}{\alpha!} f^{(\alpha)}(\vb{x}_0)\right) \nonumber\\
                & \qquad-\dfrac{1}{2h}\left(\sum_{|\alpha| \ge 0} h^{\alpha} \dfrac{\left(-1+\vb{v}_j\right)^{\alpha}}{\alpha!} f^{(\alpha)}(\vb{x}_0)\right)
            \Bigg]\nonumber\\
        &=\dfrac{1}{2h} \Bigg[
                \dfrac{1}{2h}\left(\sum_{|\alpha| \ge 0} h^{\alpha} \dfrac{2^{\alpha}}{\alpha!} f^{(\alpha)}(\vb{x}_0)\right)
                -\dfrac{1}{2h}\left(\sum_{|\alpha| \ge 0} h^{\alpha} \dfrac{0^{\alpha}}{\alpha!} f^{(\alpha)}(\vb{x}_0)\right)
            \Bigg]\nonumber\\
        & \qquad-\dfrac{1}{2h} \Bigg[
                \dfrac{1}{2h}\left(\sum_{|\alpha| \ge 0} h^{\alpha} \dfrac{0^{\alpha}}{\alpha!} f^{(\alpha)}(\vb{x}_0)\right)
                -\dfrac{1}{2h}\left(\sum_{|\alpha| \ge 0} h^{\alpha} \dfrac{(-2)^{\alpha}}{\alpha!} f^{(\alpha)}(\vb{x}_0)\right)
            \Bigg]\nonumber\\
        &=\dfrac{1}{4h^2}\left[\sum_{|\alpha| \ge 0} \left(h^{\alpha} \dfrac{2^{\alpha}+(-2)^{\alpha}}{\alpha!} f^{(\alpha)}(\vb{x}_0)\right)-2f(\vb{x}_0)\right]\nonumber\\
        &\longrightarrow \dfrac{1}{4h^2}(\{0,0,4,0,\frac{4}{3},\ldots\},\beta=0)=(\{1,0,\frac{1}{3},\ldots\},\beta=2)
\end{align}
where $0^0=1$. From this it is clear that this is a $\mathcal{O}(h^2)$ approximation to $f''(x_0)$.

This can be verified by computing the stencil composition of coefficients given by \eqref{eq:C=B(A)}. In this case we have
\begin{align}\label{eq:manual_C=A(B)}
    B(A) &= \sum_j^2 \sum_i^2 a_i b_j f(\vb{x}_i+\vb{y}_j)=\frac{1}{4h^2}\left(f(x_{-2})-2f(x_0)+f(x_{2})\right)
\end{align}
where, $x_2 = x_0+2h$ and $x_{-2} = x_0-2h$. The expansion vectors of $f(x_2)$ and $f(x_{-2})$ can be written as
\begin{align}
  \label{fx2}f(x_2) &\longrightarrow (\{1, 2, 2, \frac{4}{3}, \frac{2}{3}, \ldots\},\ \beta = 0), \\
  \label{fx-2}f(x_{-2}) &\longrightarrow (\{1, -2, 2, \frac{-4}{3}, \frac{2}{3}, \ldots\},\ \beta = 0).
\end{align}
Inserting the stencil vectors into \eqref{eq:manual_C=A(B)} results in
\begin{align}\label{eq:SmallComposition}
    B(A) &= \frac{1}{4h^2}\Bigg[(\{1, 2, 2, \frac{4}{3}, \frac{2}{3}, \ldots\},\ \beta = 0 ) +
        (\{1, -2, 2, \frac{-4}{3}, \frac{2}{3}, \ldots\},\ \beta = 0) \nonumber \\
        &\qquad - 2(\{1, 0, 0, 0, \ldots\},\ \beta = 0)\Bigg]\nonumber\\
        &= \frac{1}{4h^2}(\{0, 0, 4, 0, \frac{4}{3}, \ldots\},\ \beta = 0 ) = (\{1, 0, \frac{1}{3}, \ldots\},\ \beta = 2 ),
\end{align}
which matches the previous result.

\subsubsection{Associativity}
\label{assoc}

\begin{lemma}
Stencil composition follows the rule of associativity, i.e., no matter how we compose the two stencils $A$ and $B$, with order of derivatives $\bar{a}$ and $\bar{b}$ respectively, the composed stencil $C=A(B)=B(A)$ is equal with a derivative order of $\bar{a}+\bar{b}$.
\end{lemma}
\begin{proof}
Using the stencils $A$ and $B$ previously defined in \eqref{eq:StencilA} and \eqref{eq:StencilB},
\begin{align}\label{eq:B(A)=A(B)}
    B(A) &= \sum_j b_j \sum_i a_i f^{(\beta)}(\vb{x}_i+\vb{x}_j) = \sum_j \sum_i a_i b_j f^{(\beta)}(\vb{x}_i+\vb{x}_j) = \sum_i \sum_j  a_i b_j f^{(\beta)}(\vb{x}_i+\vb{x}_j) \nonumber \\
         &= \sum_i a_i \sum_j  b_j f^{(\beta)}(\vb{x}_i+\vb{x}_j) = A(B),
\end{align}
where any derivative shift simply moved from the $A$ stencil to the $B$ stencil.
\end{proof}

As a demonstration consider the composition of the first-order accurate forward-finite difference approximations to the first and second derivatives:
\begin{align}
    f'(x_0) &\longrightarrow \frac{f(x_1) - f(x_0)}{h} + \mathcal{O}(h),\\
    f''(x_0) &\longrightarrow \frac{f(x_2) - 2f(x_1)+f(x_0)}{h^2} + \mathcal{O}(h),
\end{align}
where the $A$ stencil corresponds to $f'(x_0)$ and the $B$ stencil to $f''(x_0)$. This results in $\vb{u}=\{1,0\}$ and $\vb{v}=\{2,1,0\}$ as the associated integer vectors with weights $\vb{a}=\{1/h,-1/h\}$ and $\vb{b}=\{1/h^2,-2/h^2,1/h^2\}$, respectively.
First consider $B(A)$:
\begin{align}\label{eq:AssociativityB(A)}
    B(A) &= \sum_j^2 \sum_i^3 a_i b_j f(\vb{x}_i+\vb{x}_j)=\sum_j^2 b_j\left[\dfrac{1}{h^2}\left(f(x_{2+j})-2f(x_{1+j})+f(x_{0+j})\right)\right]\nonumber\\
    &=\dfrac{1}{h}\left[\dfrac{1}{h^2}\left(f(x_{3})-2f(x_{2})+f(x_{1})\right)\right]
        -\dfrac{1}{h}\left[\dfrac{1}{h^2}\left(f(x_{2})-2f(x_{1})+f(x_{0})\right)\right]\nonumber\\
    &=\dfrac{1}{h^3}\left(-f(x_{0})+3f(x_{1})-3f(x_{2})+f(x_{3})\right)
\end{align}
The overall result can then be obtained via the expansions for $f(x_0)$ to $f(x_3)$,
\begin{align}\label{eq:AssociativityOrder}
    B(A) &= \dfrac{1}{h^3}\Bigg[\left(-\{1,0,0,0,0,\ldots\},\beta=0\right)+3\left(\{1,1,\dfrac{1}{2},\dfrac{1}{6},\dfrac{1}{24},\ldots\},\beta=0\right)\nonumber\\
    &\qquad-3\left(\{1,2,2,\dfrac{4}{3},\dfrac{2}{3},\ldots\},\beta=0\right)+\left(\{1,3,\dfrac{9}{2},\dfrac{9}{2},\dfrac{27}{8},\ldots\},\beta=0\right)\Bigg]\nonumber\\
    &=\dfrac{1}{h^3}\left(\{0,0,1,\dfrac{3}{2},\ldots\},\beta=0\right)=\left(\{1,\dfrac{3}{2},\ldots\},\beta=3\right),
\end{align}
which corresponds to an $\mathcal{O}(h)$ approximation to $f'''(x_0)$. Derivation of this result using the method shown in \eqref{eq:FullComposition} can be found in the appendix.

Associativity can be demonstrated in this example via determining the composition $A(B)$:
\begin{align}\label{eq:AssociativityA(B)}
    A(B) &= \sum_i^3 \sum_j^2 a_i b_j f(\vb{x}_i+\vb{x}_j)
            =\sum_i^3 a_i\left[\dfrac{1}{h}\left(f(x_{i+1})-f(x_{i+0})\right)\right]\nonumber\\
            &=\dfrac{1}{h^2}\left[\dfrac{1}{h}\left(f(x_{3})-f(x_{2})\right)\right]
                -\dfrac{2}{h^2}\left[\dfrac{1}{h}\left(f(x_{2})-f(x_{1})\right)\right]
                +\left[\dfrac{1}{h}\left(f(x_{1})-f(x_{0})\right)\right]\nonumber\\
    &=\dfrac{1}{h^3}\left(-f(x_{0})+3f(x_{1})-3f(x_{2})+f(x_{3})\right),
\end{align}
which is the same as \eqref{eq:AssociativityB(A)} and will thus have the approximation and order as \eqref{eq:AssociativityOrder}.

\subsubsection{Order of Accuracy}

The rate at which the local truncation error, expressed as a function of $h$, approaches zero as $h$ approaches zero is referred to as the \emph{order of accuracy} of the method~\cite{macthesis}. In order to show the order of accuracy of the \emph{composed} stencil, we need to introduce the concept of re-targeting.
This involves moving a stencil from a target point $\vb{y}_0\neq\vb{x}_0$ to the original target point $\vb{x}_0$. This can be accomplished by taking the Taylor Series of a linear combination, $(\vb{T},\beta)$, and accounting for the additional error terms. Recalling that $t_{\alpha}$ represents the coefficient multiplying the $h^{\alpha}f^{(\alpha+\beta)}$ term of the linear combination Taylor Series, the updated series at the original target point $\vb{x}_0$ can be obtained by replacing the original derivatives $f^{(\alpha+\beta)}$ in the sum by their own Taylor Series expanded about the original target point:
\begin{align}\label{eq:Retarget}
    \sum_{|\alpha| \ge 0} t_{\alpha}h^{\alpha} &\sum_{|\delta| \ge 0} \dfrac{(\vb{y}_0-\vb{x}_0)^{\delta}}{\delta!} f^{(\alpha+\beta+\delta)}(\vb{x}_0)
        = \sum_{|\alpha| \ge 0} t_{\alpha}h^{\alpha} f^{(\alpha+\beta)}(\vb{x}_0) \nonumber \\
        & \qquad+ \sum_{|\alpha| \ge 0} t_{\alpha}h^{\alpha}\sum_{|\delta| \ge 1} \dfrac{(\vb{y}_0-\vb{x}_0)^{\delta}}{\delta!} f^{(\alpha+\beta+\delta)}(\vb{x}_0)
        \longrightarrow (\vb{T},\beta) + (\vb{C}_{\vb{y}_0\rightarrow\vb{x}_0},\beta),
\end{align}
which demonstrates that re-targeting is simply the addition of the original Taylor series with a correction series given by $(\vb{C}_{\vb{y}_0\rightarrow\vb{x}_0},\beta)$. The first non-zero element in $\vb{C}_{\vb{y}_0\rightarrow\vb{x}_0}$ will be one order higher to the first non-zero element in $\vb{T}$ due to $\|(\vb{y}_0-\vb{x}_0)^{\delta}\|\geq h$ when $\delta\geq 1$.

As a demonstration consider re-targeting the one-dimensional, second-order accurate, center-finite-difference stencil of the second derivative at the point $y_0=x_0+h$ to the point $x_0$. Recall in this case we have $\vb{T}=\{1,0,1/12,0,\ldots\}$ and $\beta=2$. Therefore, the correction can be written as
\begin{align}
    \sum_{|\alpha| \ge 0} t_{\alpha}h^{\alpha}\sum_{|\delta| \ge 1} \dfrac{(y_0-x_0)^{\delta}}{\delta!} &f^{(\alpha+2+\delta)}(\vb{x}_0)\nonumber\\
    &=h f^{(3)}(x_0)+\dfrac{1}{2}h^2 f^{(4)}(x_0)+\dfrac{1}{4}h^3 f^{(5)}(x_0)+\cdots \nonumber \\
    &\longrightarrow (\{0,1,\dfrac{1}{2},\dfrac{1}{4},\ldots\},\beta=2).
\end{align}
Following the previous statements, the first non-zero element in the correction is of one order of $h$ higher than the original expansion, which corresponds to the second location in this case.

Adding this to the original series we obtain
\begin{align}
    (\{1,0,1/12,0,\ldots\},\beta=2)+(\{0,1,\dfrac{1}{2},\dfrac{1}{4},\ldots\},\beta=2)=
    (\{1,1,\dfrac{7}{12},\dfrac{1}{4},\ldots\},\beta=2),
\end{align}
which corresponds to the the second-derivative of $f(x)$ approximated at $x_0$ but using the stencil centered at $y_0=x_0+h$. From this, re-targeting can be thought of approximating a derivative at a point \emph{away} from $x_0$ and then calculating how well that is an approximation is of the same derivative \emph{at} $x_0$.

It is now possible to determine the order of accuracy of stencil composition. Let us take our inner stencil, $A$, as defined earlier in \eqref{eq:StencilA}. Using \eqref{eq:CompactNotation}, we can write the inner stencil as,
\begin{align}
    \sum_i a_i \left(\sum_{|\alpha| \ge 0} h^{\alpha} \dfrac{\vb{u}^\alpha_i}{\alpha!} f^{(\alpha+\beta)}(\vb{x}_0)\right)= \sum_{|\alpha| \ge 0} t_{\alpha} h^{\alpha} f^{(\alpha+\beta)}(\vb{x}_0),
\end{align}
where $t_{\alpha}$ denotes the coefficient associated with $h^\alpha$ and takes into account the associated weights $a_i$.

When applying the outer-stencil, $B$, the inner stencil is being evaluated away from the target point. Therefore, the inner stencils need to be re-targeted.  Using \eqref{eq:Retarget}, the composition can be written as,
\begin{align}
    B(A) = \sum_j b_j\sum_{|\alpha| \ge 0} t_{\alpha} h^{\alpha} f^{(\alpha+\beta)}(\vb{x}_j)
    &=\sum_j b_j\sum_{|\alpha| \ge 0} t_{\alpha}h^{\alpha} \sum_{|\delta| \ge 0} \dfrac{(\vb{x}_j-\vb{x}_0)^{\delta}}{\delta!} f^{(\alpha+\beta+\delta)}(\vb{x}_0)
\end{align}
Rearranging the summations on the right hand side we can rewrite the equation above as
\begin{align}\label{eq:AccuracyComposition}
    B(A) &= \sum_{|\alpha| \ge 0} t_{\alpha}h^{\alpha} \sum_j b_j \sum_{|\delta| \ge 0} \dfrac{(\vb{x}_j-\vb{x}_0)^{\delta}}{\delta!} f^{(\alpha+\beta+\delta)}(\vb{x}_0) \nonumber \\
    &= \sum_{|\alpha| \ge 0} t_{\alpha}h^{\alpha} \sum_j b_j \sum_{|\delta| \ge 0} h^{\delta} \dfrac{\vb{v}^\delta_j}{\delta!} f^{(\alpha+\beta+\delta)}(\vb{x}_0)
    = \sum_{|\alpha| \ge 0} t_{\alpha}h^{\alpha} \sum_{|\delta| \ge 0} t_\delta h^{\delta} f^{(\alpha+\beta+\delta)}(\vb{x}_0),
\end{align}
where $t_{\delta}$ denotes the coefficient associated with $h^\delta$ and takes into account the associated weights $b_j$. We will use this result for proving the resulting order of accuracy of a \emph{composed} stencil, demonstrated below.

\begin{lemma}\label{Lemma:Accuracy}
Stencil composition of two stencils $A$ and $B$ with orders of accuracy $q_a$ and $q_b$, respectively, results in a composed stencil $C = B(A)$ with order of accuracy $q_c = min(q_a, q_b)$.
\end{lemma}
\begin{proof}
Let the inner stencil $A$ be an approximation with an order of derivative of $p_a$ and order of accuracy of $q_a$. Then, we can write $A$ as,
\begin{align}
    A = f^{(p_a)} + \sum_{|\alpha| \ge q_a} t_{\alpha} h^{\alpha} f^{(\alpha+p_a)}(\vb{x}_0)&\longrightarrow (\vb{T}_A, \beta = p_a)\nonumber\\
    &\longrightarrow(\{1,0,\cdots,0,t_{q_a},\ldots\}, \beta = p_a),
\end{align}
where $t_{q_a}$ is the coefficient associated with order of accuracy term $h^{q_a}$. Similarly, we can write the outer stencil $B$ approximating an order of derivative of $p_b$ and order of accuracy $q_b$ as
\begin{align}
    B \longrightarrow (\vb{T}_B, \beta = p_b)\longrightarrow(\{1,0,\cdots,0,t_{q_b},\ldots\}, \beta = p_b).
\end{align}

From \eqref{eq:AccuracyComposition}, we can write the composition $B(A)$ as,
\begin{align*}
    B(A) = \sum_{|\alpha| \ge 0} t_{\alpha}h^{\alpha} \sum_{|\delta| \ge 0} t_\delta h^{\delta} f^{(\alpha+\beta+\delta)}(\vb{x}_0).
\end{align*}
Using coefficient vectors $\vb{T}_A$ and $\vb{T}_B$ the same equation can be written as,
\begin{align}
    B(A) &\longrightarrow (\vb{T}_A, \beta = p_a) \circ (\vb{T}_B, \beta = p_b)\longrightarrow (\vb{T}_A\circ\vb{T}_B,\beta=p_a+p_b).
\end{align}
Recall that $\vb{T}_A$ and $\vb{T}_B$ are simply short-hand notation for
\begin{align}
    \vb{T}_A&\longrightarrow \{1\cdot h^0,0\cdot h^1,\cdots,0\cdot h^{q_a-1}, t_{q_a} h^{q_a},\ldots\},\\
    \vb{T}_B&\longrightarrow \{1\cdot h^0,0\cdot h^1,\cdots,0\cdot h^{q_b-1}, t_{q_b} h^{q_b},\ldots\}.
\end{align}
Then the list composition can be written as
\begin{align}
    \vb{T}_A\circ\vb{T}_B\longrightarrow &1h^0(\{1h^0,0,\cdots,0,t_{q_b}h^{q_b},\ldots\})+0h^1(\{1h^0,0,\cdots,0,t_{q_b}h^{q_b},\ldots\})+\cdots\nonumber\\
    &+t_{q_a}h^{q_a}(\{1h^0,0,\cdots,0,t_{q_b}h^{q_b},\ldots\})+\ldots
\end{align}

If $q_a < q_b$, the first non-unitary element in $B(A)$ would be, $t_{q_a} h^{q_a}$, implying an order of accuracy of $q_a$. Similarly if $q_b < q_a$, the first non-zero term would be, $t_{q_b} h^{q_b}$, and hence order accuracy being $q_b$. In the case of $q_a = q_b$, the first non-zero term would be $t_{q_a} h^{q_a}+t_{q_b} h^{q_b}$ and since $q_a = q_b$ this implies the order of accuracy is $q_a = q_b$. This proves that when two stencils with orders of accuracy $q_a$ and $q_b$ respectively are composed, the order of accuracy of the composed stencil is $\textnormal{min}(q_a, q_b)$.
\end{proof}

Another observation can be made here regarding the coefficient of the leading order error term in the composed stencil, written as a Proposition below.
\begin{proposition}\label{Prop:CoeffAccuracy}
Stencil composition of two stencils $A$ and $B$ with orders of accuracy $q_a$ and $q_b$, respectively, with coefficients of the leading order error terms being $t_{q_a}$ and $t_{q_b}$, respectively, leads to a composed stencil $C = B(A)$ with order of accuracy $q_c=min(q_a, q_b)$ and a leading-order error coefficient $t_{q_c}$ equal to,
\[
    t_{q_c} = \left\{\begin{array}{ll}
    t_{q_a} & \text{if } q_c = q_a,\\
    t_{q_b} & \text{if } q_c = q_b,\\
    t_{q_a}+t_{q_b}, & \text{if } q_c = q_a = q_b.
    \end{array}\right.
\]
\end{proposition}

As an example, we compose a second-order accurate first derivative stencil with a fourth-order accurate first derivative stencil to obtain a second-derivative stencil. As a reminder, the second-order accurate first derivative stencil \eqref{f'stencil} is
\begin{equation*}
  f'(x_0) \longrightarrow \frac{f(x_1) - f(x_{-1})}{2h} + \mathcal{O}(h^2)\longrightarrow(\{1, 0, \frac{1}{6}, 0, \ldots\},\ \beta = 1),
\end{equation*}
and the fourth-order accurate first derivative stencil can be written as
\begin{align}\label{eq:4orderf'}
    f'(x_0) &\longrightarrow \frac{8f(x_1) - f(x_2) + f(x_{-2}) - 8f(x_{-1})}{12h} + \mathcal{O}(h^4)\nonumber \\
            &\longrightarrow\left(\{1,0,0,0,\frac{-1}{30},\ldots\},\beta=1\right).
\end{align}

Let the outer-stencil correspond to second-order accurate approximation to $f'(x_0)$ while the inner-stencil is the fourth-order accurate approximation to $f'(x_0)$. This results in $\vb{u}=\{1,2,-2,-1\}$ and $\vb{v}=\{1,-1\}$ as the associated integer vectors with weights $\vb{a}=\{8/12h$, $-1/12h$, $1/12h$, $-8/12h\}$ and $\vb{b}=\{1/h,-1/h\}$, respectively.
Performing the composition we get
\begin{align}
    B(A) &= \sum_j^2 \sum_i^4 a_i b_j f(\vb{x}_i+\vb{x}_j)\nonumber\\
    &=\sum_j^2 b_j\left[\dfrac{1}{12h}\left(8f(x_{1+j})-f(x_{2+j})+f(x_{-2+j})-8f(x_{-1+j})\right)\right]\nonumber\\
    &=\dfrac{1}{h}\left[\dfrac{1}{12h}\left(8f(x_{2})-f(x_{3})+f(x_{-1})-8f(x_{0})\right)\right] \nonumber \\
        &\qquad-\dfrac{1}{h}\left[\dfrac{1}{12h}\left(8f(x_{0})-f(x_{1})+f(x_{-3})-8f(x_{-2})\right)\right]\nonumber\\
    &=\dfrac{1}{12h^2}\left(-f(x_{-3})+8f(x_{-2})+f(x_{-1})-16f(x_{0})+f(x_{1})+8f(x_{2})-f(x_{3})\right)
\end{align}

The overall result can then be obtained via the expansions for $f(x_{-3})$ to $f(x_3)$,
\begin{align}
    B(A) &= \dfrac{1}{12h^2}\Bigg[-\left(\{1,3,\dfrac{9}{2},\dfrac{9}{2},\dfrac{27}{8},\ldots\},\beta=0\right)+8\left(\{1, -2, 2, \frac{-4}{3}, \frac{2}{3}, \ldots\},\beta=0\right)\nonumber\\
    &\qquad+\left(\{1, -1, \frac{1}{2}, \frac{-1}{6}, \frac{1}{24} \ldots\},\beta=0\right) -16\left(\{1,0,0,0,0,\ldots\},\beta=0\right)\nonumber\\
    &\qquad+\left(\{1, 1, \frac{1}{2}, \frac{1}{6}, \frac{1}{24} \ldots\},\beta=0\right)+8\left(\{1, 2, 2, \frac{4}{3}, \frac{2}{3}, \ldots\},\beta=0\right)\nonumber\\
    &\qquad-\left(\{1,-3,\dfrac{9}{2},\dfrac{-9}{2},\dfrac{27}{8},\ldots\},\beta=0\right)\Bigg] \nonumber\\
    &=\dfrac{1}{12h^2}\left(\{0,0,12,0,2,\ldots\},\beta=0\right)=\left(\{1,0,\dfrac{1}{6},\ldots\},\beta=2\right),
\end{align}

Upon composing a fourth-order accurate stencil with a second order accurate stencil the leading order error term is in the location corresponding to $h^2$ and thus the composition order of accuracy is $\textnormal{min}(2,4)$, demonstrating Lemma \ref{Lemma:Accuracy}. This also demonstrates Proposition \ref{Prop:CoeffAccuracy}, as the coefficient of error term in the composed stencil is $\dfrac{1}{6}$ which matches the coefficient of the error term in second-order accurate stencil in \eqref{f'stencil}.

\subsubsection{Stability}

The stability of a finite difference approximation to a differential equation is as important as the accuracy. Unlike for accuracy, it is not possible to construct a generalized theorem for the stability of stencil composition from the stability of the inner stencils. Instead, sample cases will be considered and the general stability of stencil composition will be compared to the compact case.

Consider the stability of the heat equation for $f(x,t)$, $\partial_t f=\partial_{xx} f$. This can be discretized via Eq.~\eqref{eq:manual_C=A(B)} and a first-order discretization in time about the point $x_0$:
\begin{equation}
    \dfrac{f_0^{n+1} - f_0^n}{\Delta t} = \dfrac{f_{-2}^n - 2 f_{0}^n + f_{+2}^n}{4 h^2},
\end{equation}
where $n$ refers to time $t_n$ and $n+1$ to time $t_n+\Delta t$, $f_0^n = F(x_0, t_n)$, and $f_{\pm 2}^n = f(x_0\pm 2h,t_n)$. Assume that the solution for the $n^{th}$ time step is $f_j^n=\xi_k^n e^{\imath k j h}$ where $k$ is the wave mode and $\xi_k$ is the growth factor. Using this in the discretization, dividing by $f_0^n$ results, and solving for $\xi_k$ results in
\begin{equation}
    \xi_k = 1 + \dfrac{\Delta t}{2h^2}\left(\cos{2hk}-1\right).
\end{equation}
Stability requires that $|\xi_k|\leq 1$ and thus
\begin{align}
    -1 & \leq 1 + \dfrac{\Delta t}{2h^2}\left(\cos{2hk}-1\right) \leq 1 \nonumber\\
    -2 & \leq \dfrac{\Delta t}{2h^2}\left(\cos{2hk}-1\right) \leq 0 \nonumber\\
    -4 & \leq \dfrac{\Delta t}{h^2}\left(\cos{2hk}-1\right) \leq 0.
\end{align}
As $\left(\cos{2hk}-1\right)\in [-2,0]$ we have
\begin{equation}
    0 \leq \dfrac{\Delta t}{h^2} \leq 2.
\end{equation}
Thus, if $\Delta t \leq 2 h^2$ holds the method is considered stable. This compares to a requirement of $\Delta t\leq h^2/2$ for the compact version using locations $x_{-1}$, $x_0$, and $x_{+1}$. This should be expected as the distance between points in the stencil is twice that for the compact stencil.

Common compositions of higher-order derivatives, such as approximating $f_{xxxx}$ using 
$\partial_{xx}\left(\partial_{xx}f\right)$ 
and center-finite differences, results in the same stencils as compact schemes. To demonstrate this and to explore any changes in stability consider the stability requirement for the fourth-order PDE $\partial_t f = -\partial_x\left(M\partial_x\left(\kappa\partial_{xx}f\right)\right)$, where $M$ is a (potentially varying) mobility and $\kappa$ is a gradient energy coefficient. This equation a simplified version of the Cahn-Hilliard equation explored in Sec.~\ref{sec:cahn}, neglecting the chemical free energy. Typically, the gradient energy coefficient is a constant and thus for simplicity $\kappa=1$ in this example. In the case where $M$ is a constant this can be discretized directly from $\partial_t f=-\partial_{xxxx}f$. When $M$ is spatially varying, though, it is easier to implement the method through stencil composition whereby $\partial_x\left(\partial_{xx}\right)$ is composed first, scaled by $M$ at the grid location, and that result is then composed with $\partial_x$. Again, for simplicity let use consider $M=1$, but compare the stability requirement for a direct discretization versus a composed one, summarized below. Full details of the stability requirement derivation is given in Appendix~\ref{sec:appendixB}.

\newcolumntype{C}[1]{>{\centering\arraybackslash}p{#1}}
\renewcommand{\arraystretch}{1.25}
\begin{table}[H]
    \centering
    \caption{Stability requirement on the time-step, $\Delta t\leq \alpha h^4$, for $\partial_t f=-\partial_x\left(\partial_x\left(\partial_{xx}f\right)\right)$.}
    \begin{tabular}{ C{4cm} | C{2cm} | C{2cm} }
    \toprule
     & $\alpha$ & Support \\
    \midrule
      Compact  & $1/8$ & $[-2,2]$ \\ 
      Composed: $\partial_{xx}\left(\partial_{xx}\right)$ & 1/8 & $[-2,2]$ \\
      Composed: $\partial_x\left(\partial_x\left(\partial_{xx}\right)\right)$ & 27/32 & $[-3,3]$ \\
    \bottomrule
    \end{tabular}
\end{table}

As can be seen the stability requirement for the composed stencil of $\partial_x\left(\partial_x\left(\partial_{xx}\right)\right)$ has a stable time step that is 6.75 times greater than the compact scheme. This should be expected as the support is wider. While this does result in a slightly higher memory footprint when implemented, it does allow for easily incorporating variable mobility as mentioned previously.

%
%

\subsection{Higher-dimensional Stencil Composition}
\label{sec:2dcomp}

Another benefit of using stencil composition to obtain higher order derivatives is its possibility of extension to higher dimensions. By performing stencil composition in two or three dimensions one can easily obtain higher-order mixed derivatives. When two stencils in different dimensions are composed the resulting stencil is simply the outer product of the two stencils, demonstrated using two-dimensional examples below.

If we have two stencils $\vb{T}_{X_i}$ and $\vb{T}_{Y_j}$ in the $x$ and $y$ directions, respectively, approximating derivative of orders $p_x$ and $p_y$, such that
\begin{align}
  f^{(p_x)}(\vb{x}_i) \longrightarrow \left(\vb{T}_{X_i}, \ \beta_x = p_x\right)\ \text{and}\ f^{(p_y)}(\vb{y}_j) \longrightarrow \left(\vb{T}_{Y_j}, \ \beta_y = p_y \right),
\end{align}
then the composition of the two stencils yields a mixed derivative equal to the outer product of the two stencils,
\begin{align}
  f^{(p_x,\ p_y)}(\vb{x}_i,\ \vb{y}_j) \longrightarrow (\vb{T}_{X_i} \otimes \vb{T}_{Y_j}, \ \beta_x=p_x,\ \beta_y=p_y)
\end{align}

As an example, let us compose first derivative stencils in the $x$ and $y$ direction to obtain the mixed derivative $f^{(1,1)}(x,y)$. Let $\vb{x}_0=(x_0,y_0)$ be the target point. We can write the stencils of second-order accurate first derivative stencils in each direction thusly,
\begin{align}
  f'(x_0) &\longrightarrow (\{1, 0, \frac{1}{6}, 0, \ldots\},\ \beta_x = 1), \\
  f'(y_0) &\longrightarrow (\{1, 0, \frac{1}{6}, 0, \ldots\},\ \beta_y = 1).
\end{align}
Taking the outer product we obtain
\begin{align}
    (\{1, 0, \frac{1}{6}, 0, \ldots\} \otimes \{1, 0, \frac{1}{6}, 0, \ldots\},\ \beta_x = 1,\ \beta_y = 1)=&\nonumber \\
        \left(\left[
        \begin{array}{ccccc}
        1 & 0 & \frac{1}{6} & 0 &\cdots \\
        0 & 0 & 0 & 0 &\cdots \\
        \frac{1}{6} & 0 & \frac{1}{36} & 0 &\cdots \\
        0 & 0 & 0 & 0 &\cdots \\
        \vdots & \vdots & \vdots & \vdots & \ddots
        \end{array}
        \right],\ \beta_x = 1,\ \beta_y = 1\right),&
\end{align}
which can be written in the expanded form as the following,
\begin{align}
    f^{(1,1)}(\vb{x}_i,\ \vb{y}_j) \longrightarrow& f^{(1,1)}(x_0,y_0) + \frac{1}{6}h_x^2\cdot f^{(3,1)}(x_0,y_0) + \frac{1}{6}h_y^2\cdot f^{(1,3)}(x_0,y_0)\nonumber\\
    &\qquad+ \frac{1}{36}h_x^2h_y^2\cdot f^{(3,3)}(x_0,y_0) + \ldots,
\end{align}
where $h_x$ and $h_y$ are the spacing in the lattice in $x$ and $y$ directions respectively. For simplicity, assume that $h_x = h_y = h$. We can then rewrite the equation above as following,
\begin{align}
    f^{(1,1)}(\vb{x}_i,\ \vb{y}_j) \longrightarrow& f^{(1,1)}(a,b) + \frac{1}{6}h^2\cdot f^{(3,1)}(a,b) + \frac{1}{6}h^2\cdot f^{(1,3)}(a,b)\nonumber\\
        &\qquad+ \frac{1}{36}h^4\cdot f^{(3,3)}(a,b) + \ldots.
\end{align}
Note that the error term is proportional to $\mathcal{O}(h^2)$, therefore the composed stencil is also second-order accurate.

Lets us now compose two stencils of different order of accuracies in two dimensions. For instance, composition of a fourth-order accurate stencil in the x-direction and second-order accurate stencil in the y-direction,
\begin{align}
  f'(a) &\longrightarrow (\{1, 0, 0, 0, \frac{-1}{30}, 0, \ldots\},\ \beta_x = 1), \\
  f'(b) &\longrightarrow (\{1, 0, \frac{1}{6}, 0, \frac{1}{120}\ldots\},\ \beta_y = 1),
\end{align}
will give us the following,
\begin{align}
    f^{(1,1)}(a,b)\longrightarrow
        \left(\left[
        \begin{array}{cccccc}
        1 & 0 & \frac{1}{6} & 0 & \frac{1}{120} &\cdots \\
        0 & 0 & 0 & 0 & 0 &\cdots \\
        0 & 0 & 0 & 0 & 0 &\cdots \\
        0 & 0 & 0 & 0 & 0 &\cdots \\
        \frac{-1}{30} & 0 & \frac{-1}{180} & 0 & \frac{-1}{3600} &\cdots \\
        \vdots & \vdots & \vdots & \vdots &\vdots & \ddots
        \end{array}
        \right],\ \beta_x = 1,\ \beta_y = 1\right).
\end{align}
In this case the highest order error term would be $\frac{1}{6}h^2\cdot f^{(1,3)} \in \mathcal{O}(h^2)$, demonstrating that the composed stencil is second-order accurate. Both of the examples illustrate that Lemma \ref{Lemma:Accuracy} still holds true for higher dimensional composition.

%
%

\section{Numerical Examples}

In this section the convergence of one- and two-dimensional examples, in addition to the bi-harmonic equation is presented. In addition to verification of the expected order-of-accuracy, we will also determine the coefficient associated with error.

%
%
\subsection{One-dimensional Example}
\label{sec:1dex}

\begin{figure}
    \centering
    \includegraphics[width=0.75\textwidth]{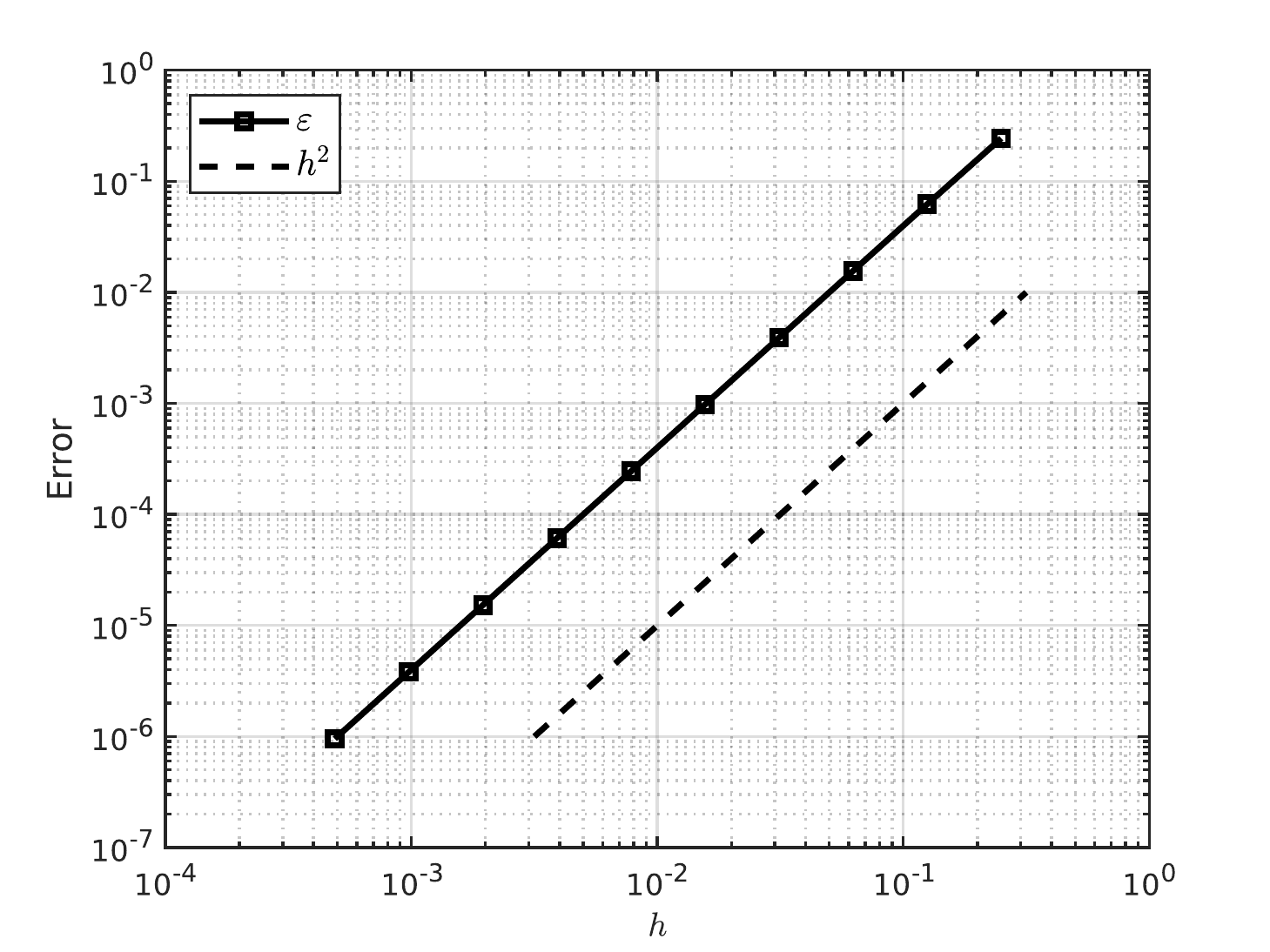}
    \caption{The error evaluated at $x=\pi$ between the third-derivative of $f(x) = \sin(x)\cos(x)$ and a finite-difference stencil obtained via composition of second-order accurate approximations to the first- and second-derivatives.}
    \label{fig1}
\end{figure}

Begin by considering the one-dimensional function $f(x) = \sin(x)\cos(x)$. We will approximate the third-derivative of this function, $f^{(3)}(x) = 4\sin^2(x) -4\cos^2(x)$, via the composition of second-order accurate first-derivative and second derivative stencils given by $\vb{u}=\{1,-1\}$ and $\vb{v}=\{1,0,-1\}$ with weights of $\vb{a}=\{1/h,-1/h\}$ and $\vb{b}=\{1/h^2,-2/h^2,1/h^2\}$, respectively. This results in a series for the first-derivative of
\begin{align}
    f'(x)\longrightarrow(\{1,0,\dfrac{1}{6},0,\dfrac{1}{120},\ldots\},\beta=1)
\end{align}
while the second-derivative series is
\begin{align}
    f''(x)\longrightarrow(\{1,0,\dfrac{1}{12},0,\dfrac{1}{360},\ldots\},\beta=2).
\end{align}
We expect that the composition will result in an $\mathcal{O}(h^2)$-accurate stencil with a leading-order error coefficient of $1/6+1/12=1/4$, or in other words we expect that the error will scale as $\tfrac{1}{4}h^2$, which can be verified from the Taylor-Series of the composition:
\begin{align}
   f^{(3)}(x)\longrightarrow(\{1,0,\dfrac{1}{4},0,\dfrac{1}{40},\ldots\},\beta=3).
   \label{eq:Ex1Taylor}
\end{align}

The error evaluated at $x=\pi$ as a function of grid-spacing $h$ is shown in Fig.~\ref{fig1}. As expected, the rate-of-convergence equals that of the prediction. To verify the coefficient associated with this convergence we fit a line in log-log space:
\begin{align}
    \label{eq:fit}
    \log\varepsilon = p\log h + \log C,
\end{align}
where $\varepsilon$ is the error, $p$ is the calculated order of convergence and $C$ is the leading-order coefficient. Fitting the data results in $p=1.9976\approx 2$, which matches the expected order of convergence and $C\approx 3.9432$. From the Taylor-Series of the approximation, \eqref{eq:Ex1Taylor}, the coefficient of the error should equal $\tfrac{1}{4}f^{(5)}(\pi)=4$, which is very close to the calculated value.

%
%
\subsection{Two-dimensional Example}
\label{sec:2dex}

\begin{figure}
    \centering
    \includegraphics[width=0.75\textwidth]{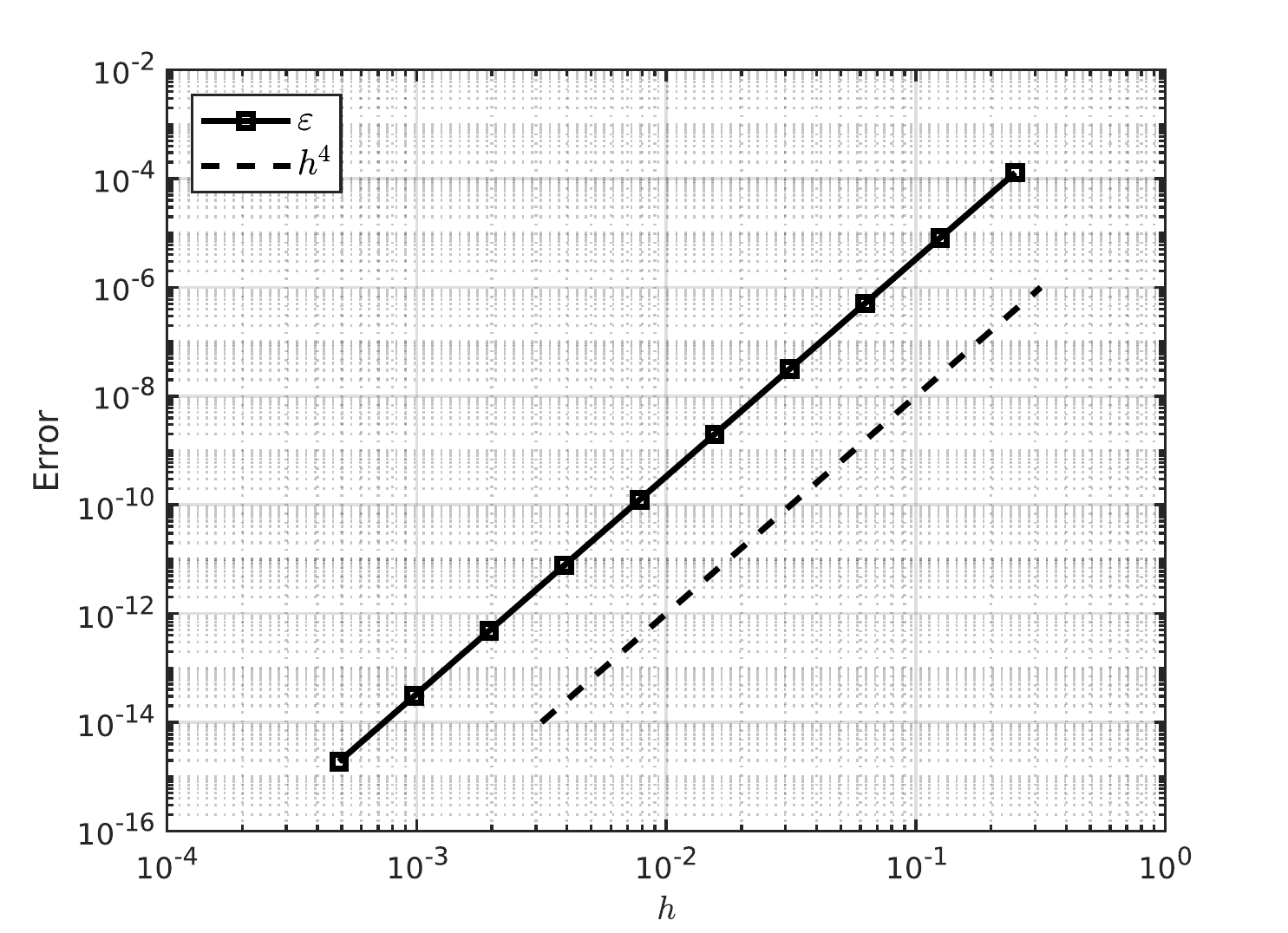}
    \caption{The error of approximating $f^{(4,3)}(x,y)$, where $f(x,y) = \sin(x)\cos(y) + \cos(x)\sin(y)$, obtained by composing two fourth-order accurate second derivative stencils in $x$ and fourth-order accurate first and second derivative stencils in $y$, evaluated at $(x,y) = (2\pi,\pi/3)$. The dashed lines indicate a fourth-order accurate solution.}
    \label{fig2}
\end{figure}

Let us now consider a two-dimensional function, \linebreak $f(x,y) = \sin(x)\cos(y) + \cos(x)\sin(y)$ where we are interested approximating $f^{4,3}(x)$ with fourth-order accuracy. As before we will verify the order of accuracy and the associated coefficient. To build the overall stencil we will be using multiple compositions. First, a centered, fourth-order accurate discretization of the second-derivative is composed with itself to obtain a fourth-order accurate representation of the fourth-derivative:
\begin{align}
    f^{(4)}(\vb{x})\longrightarrow (\{1,0,0,0,-\dfrac{1}{45},0,-\dfrac{1}{504},\ldots\},\beta=4).
\end{align}
Second, a centered fourth-order accurate discretization of the first-derivative is composed with the fourth-order accurate second derivative approximation:
\begin{align}
    f^{(3)}(\vb{x})\longrightarrow (\{1,0,0,0,-\dfrac{2}{45},0,-\dfrac{5}{1008},\ldots\},\beta=3).
\end{align}
Composing the fourth-derivative in the $x$-direction with the third-derivative in the $y$-direction results in
\begin{align}
    f^{(4,3)}(\vb{x})\longrightarrow
        \left(\left[
        \begin{array}{cccccc}
        1 & 0 & 0 & 0 & -\frac{2}{45} &\cdots \\
        0 & 0 & 0 & 0 & 0 &\cdots \\
        0 & 0 & 0 & 0 & 0 &\cdots \\
        0 & 0 & 0 & 0 & 0 &\cdots \\
        -\frac{1}{45} & 0 & 0 & 0 & \frac{2}{2025} &\cdots \\
        \vdots & \vdots & \vdots & \vdots &\vdots & \ddots
        \end{array}
        \right],\ \beta_x = 4,\ \beta_y = 3\right),
\end{align}
which corresponds to $f^{4,3}(\vb{x})+h^4\left(-\frac{1}{45} f^{(8,3)}(\vb{x})-\frac{2}{45} f^{(4,7)}(\vb{x})\right)+\cdots$.

The error between the exact solution, $f^{(4,3)}(x, y) = \sin(x)\sin(y) - \cos(x)\cos(y)$, and the approximation computed via the composed finite-difference stencil at location $\vb{x}=(2\pi,\pi/3)$ is shown in Fig.~\ref{fig2}. Following \eqref{eq:fit} the calculated convergence rate is 3.9997 while $C=3.324\times 10^{-2}$, close to the expected value of
$-\left( f^{(8,3)}(\vb{x})+ 2f^{(4,7)}(\vb{x})\right)/45=1/30$.

%
%

\subsection{Biharmonic Equation}
\label{sec:biharmonic}

\begin{figure}
    \centering
    \includegraphics[width=0.75\textwidth]{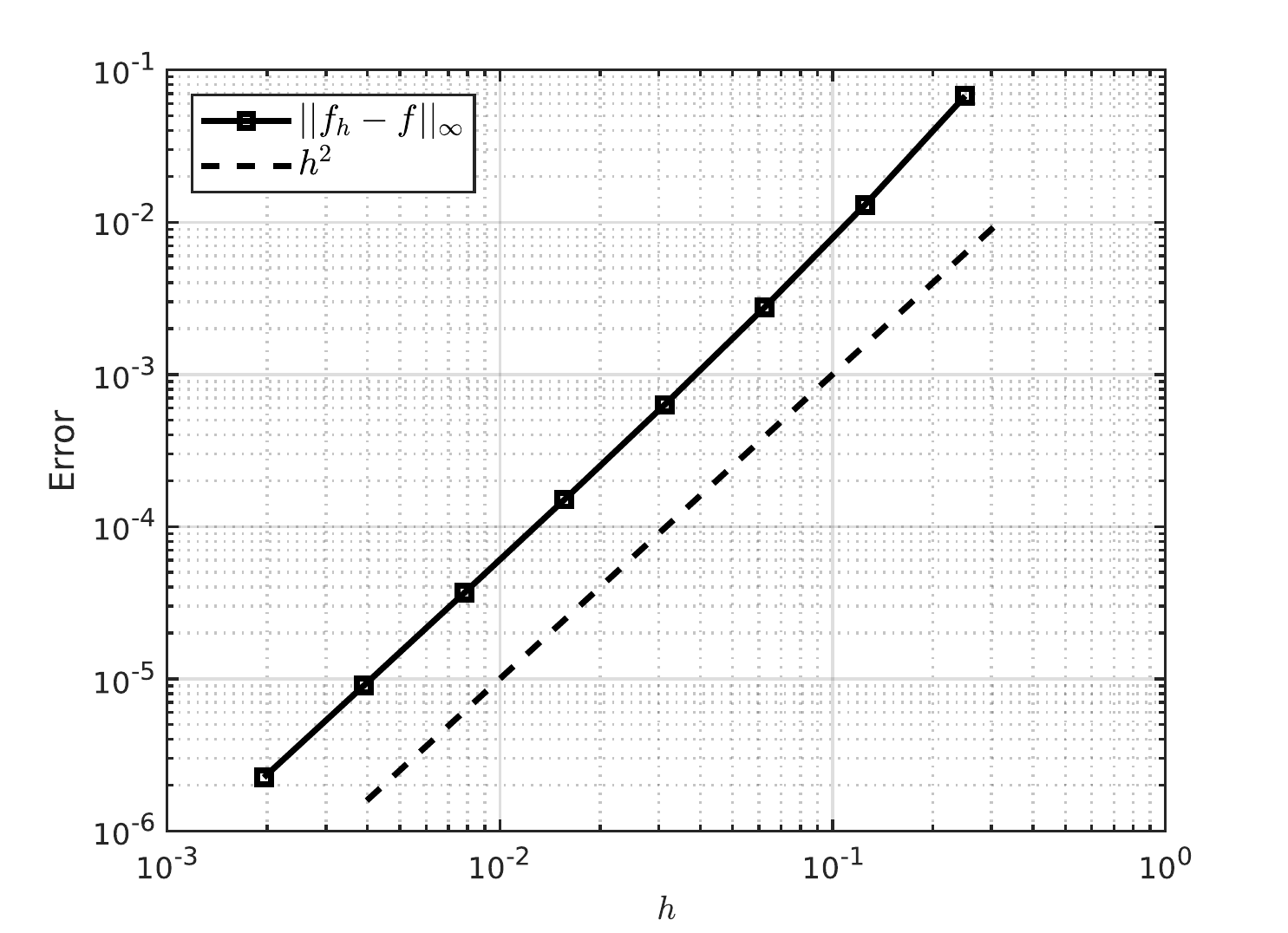}
    \caption{Convergence test of biharmonic boundary value problem. Show is the $l_\infty$-norm of the error in the domain versus grid spacing $h$ for $\Delta(\Delta f) = 4\pi^4 \sin(\pi x) \sin(\pi y)$ obtained by composition of second-order accurate stencils evaluated on a unit square domain $\{0\leq x\leq 1, 0\leq y\leq 1\}$ with appropriate boundary conditions. The dashed line indicates slope for a second-order accurate solution, matching the expected result.}
    \label{fig3}
\end{figure}

Next, we consider the solution of a linear system arising from the discretization of a high-order differential equation. Specifically, we consider solutions of the biharmonic equation, which is a fourth-order linear partial differential equation with applications in various areas of mechanics, including the theory of elasticity and flow of viscous fluids~\cite{APS}. In two-dimensions, the biharmonic of a function $f(x,y)$ can be written as
\begin{equation}\label{biharm}
    \Delta(\Delta f)=\frac{\partial^4 f}{\partial x^4} + 2\frac{\partial^4 f}{\partial x^2\partial y^2} + \frac{\partial^4 f}{\partial y^4} = g(x,y),
\end{equation}
with appropriate boundary conditions on a bounded domain~\cite{ford} and where $g(x,y)$ is the problem-specific forcing function.

For our numerical experiment we consider a simply supported rectangular plate with sides of unit length $\{0\leq x\leq 1, 0\leq y\leq 1\}$ and a given solution of $f(x,y) = \sin(\pi x) \sin(\pi y)$. This results in boundary conditions of~\cite{arad}
\begin{subequations}\label{bc}
    \begin{align}
        f &= 0, \frac{\partial^2 f}{\partial x^2} = 0\quad \text{for}\ x = 0\ \text{and}\ x = 1, \textnormal{and} \\
        f &= 0, \frac{\partial^2 f}{\partial y^2} = 0\quad \text{for}\ y = 0\ \text{and}\ y = 1
    \end{align}
\end{subequations}
and a forcing function of
\begin{align}
    g(x,y) = 4\pi^4\sin(\pi x)\sin(\pi y).
\end{align}

There are two ways we can make use of composition to obtain the stencil for the biharmonic equation. In the first method we can use composition to discretize the middle equation of \eqref{biharm}. While this is straight-forward to accomplish, it requires that the spatial dimension of the underlying grid be taken into account as there will be additional terms in the $z$-direction if this is a three-dimensional problem instead of a two-dimensional one. An alternative is to create a single Laplacian stencil and compose this stencil with itself, \textit{i.e.} using stencil composition to compute the left-hand side equation of \eqref{biharm} directly. From an implementation point-of-view this second approach is much more attractive as any dimension-dependence will already be taken into account when forming the Laplacian stencil. Additionally, as both methods will result in the same stencil, so the second approach is the one used here.

The error results in the $l_\infty$-norm for the given problem are shown in Fig.~\ref{fig3} where the discretization of the Laplacian was achieved using second-order accurate stencils. Based on this, we expect that the discretization of the biharmonic equation will maintain this second-order accuracy, which is verified by calculating the rate-of-convergence of the test.

\section{Cahn-Hilliard Phase-Field Model: A Benchmark Problem}
\label{sec:cahn}
In this section we look at a benchmark phase-field problem involving spinodal decomposition in a binary system which uses the Cahn-Hilliard equation for time evolution. This example is inspired from the first benchmark problem in~\cite{spinodal}. For more information on spinodal decomposition, Cahn-Hilliard equations, and other relevant benchmark problems please refer to~\cite{cahn, spinodal}. 

For this example, various studies and analyses are done, such as temporal convergence test, scaling analysis as well as investigating matrix properties and stability. A two dimensional (2D) and a three-dimensional (3D) computational domain have been used for the simulation, and the initial conditions are chosen accordingly. The following experiment has been conducted using \petscfd, a finite-difference discretization class in \petsc~\cite{petsc}. \petsc is a library of data structures and routines that allows implementation of large-scale application codes on parallel as well as serial computers. For discretization, \petsc primarily uses \code{PetscFE} or \code{PetscFV} for finite element and finite volume based discretizations respectively. Therefore, \petscfd, which leverages the concept of \scomp, adds the support for finite difference based discretizations in \petsc for solving PDEs. Since \petscfd is a class in \petsc, a large amount of software complexity is avoided. For more information regarding \petscfd and its usage, refer to~\cite{mishra}.

\subsection{Computational Domain, Free Energy and Dynamics}
\label{sec:domain}
For spinodal decomposition in a binary system, a single order parameter, $c$, is evolved, which describes the atomic fraction of solute~\cite{spinodal}. The free energy of the system, $F$, is expressed as~\cite{cahn}
\begin{equation}
    F(c) = \int\left(f_{chem}(c) + \frac{\kappa}{2}|\nabla c|^2\right) dV,
\end{equation}
where $f_{chem}$ is the chemical free energy density and $\kappa = 2$ is the gradient energy coefficient.  $f_{chem}$ has a simple polynomial form,
\begin{equation}
    f_{chem}(c) = \rho(c-c_\alpha)^2(c-c_\beta)^2,
\end{equation}
such that $f_{chem}$ is a symmetric double-well with minima at $c_\alpha = 0.3$ and $c_\beta = 0.7$, while $\rho = 5$ controls the height of the double-well barrier.

The evolution of $c$ is given by the Cahn-Hilliard equation~\cite{cahn}
\begin{equation}\label{eq:cahn}
    \frac{\partial c}{\partial t} = \nabla\cdot \left\{M\nabla\left(\frac{\partial f_{chem}}{\partial c} - \kappa\Delta c\right)\right\},
\end{equation}
where $M = 5$ is the mobility of the solute.

For this problem, we consider a two-dimensional and a three-dimensional computational domain. The two-dimensional domain is of size $200\times200$ units, centered at $x = y = 100$, and the three-dimensional domain is of size $64\times64\times64$ units, centered at $x = y = 32$. For both 2D and 3D, periodic boundary conditions are assumed on all boundaries. The initial conditions for this problem are chosen such that the average value of $c$ over the computational domain is approximately 0.5. The initial value of $c$ for 2D and 3D domains is given by
\begin{align}
    c(x,y) = c_0 &+ \epsilon[\cos(0.105x)\cos(0.11y) \nonumber\\
    &+ [\cos(0.13x)\cos(0.087y)]^2 \\
    &+ \cos(0.025x-0.15y)\times \cos(0.07x-0.02y)],\nonumber
\end{align}
\begin{align}
    c(x,y,z) = c_0 &+ \epsilon[\cos(0.105x)\cos(0.11y)\cos(0.11z) \nonumber \\
    &+ [\cos(0.13x)\cos(0.087y)\cos(0.1z)]^2 \\
    &+ \cos(0.025x-0.15y-0.1z) \times \cos(0.07x-0.02y+0.01z)],\nonumber
\end{align}
where $c_0 = 0.5$ and $\epsilon = 0.01$.
The system is discretized by composing the Laplacian with itself. The linear portion, $\Delta^2 c$, is treated implicitly while the non-linear portion is treated explicitly. 

\begin{figure}
\begin{subfigure}{0.49\textwidth}
\includegraphics[width=\linewidth]{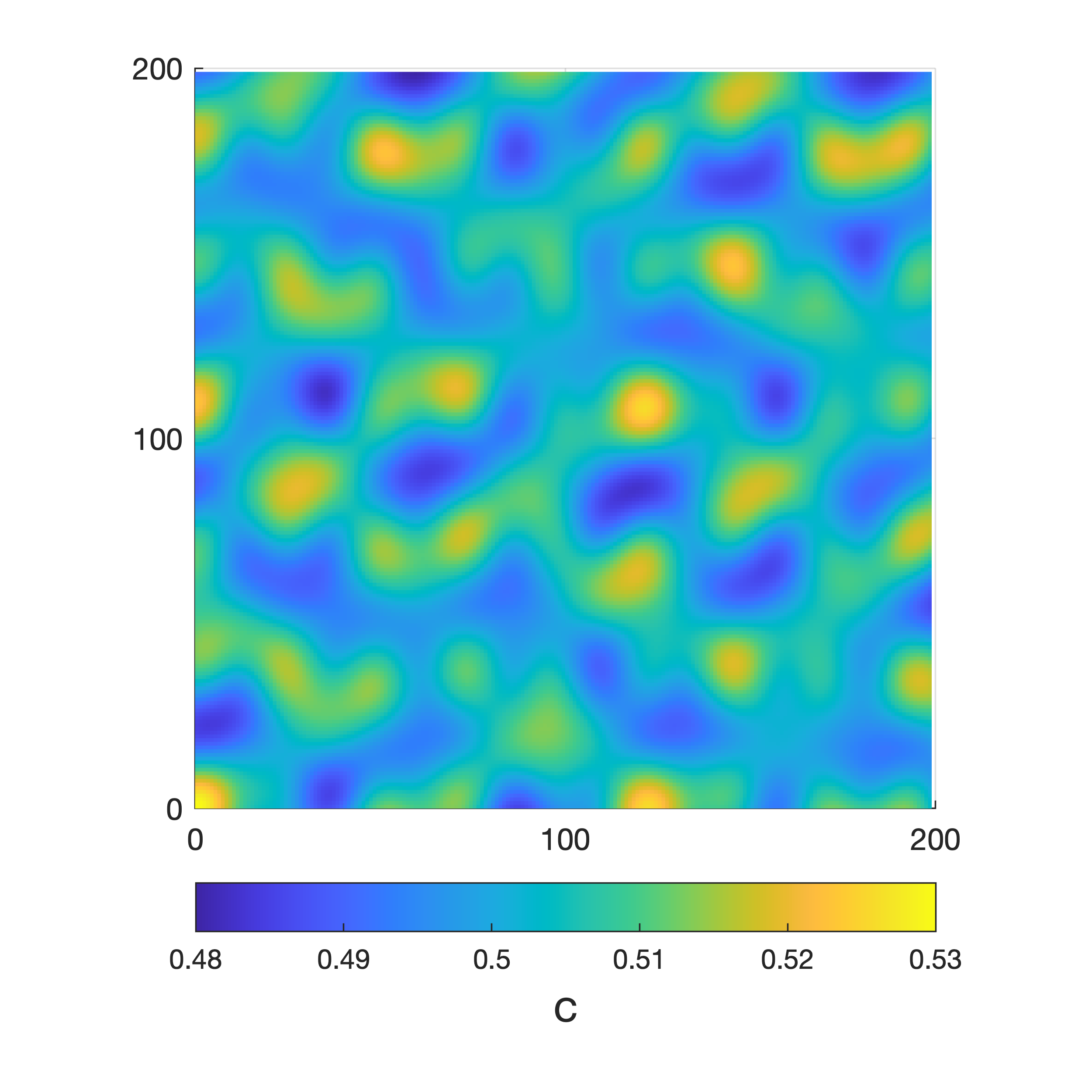}
\caption{2D: $200\times200$}
\end{subfigure}\hspace*{\fill}
\begin{subfigure}{0.49\textwidth}
\includegraphics[width=\linewidth]{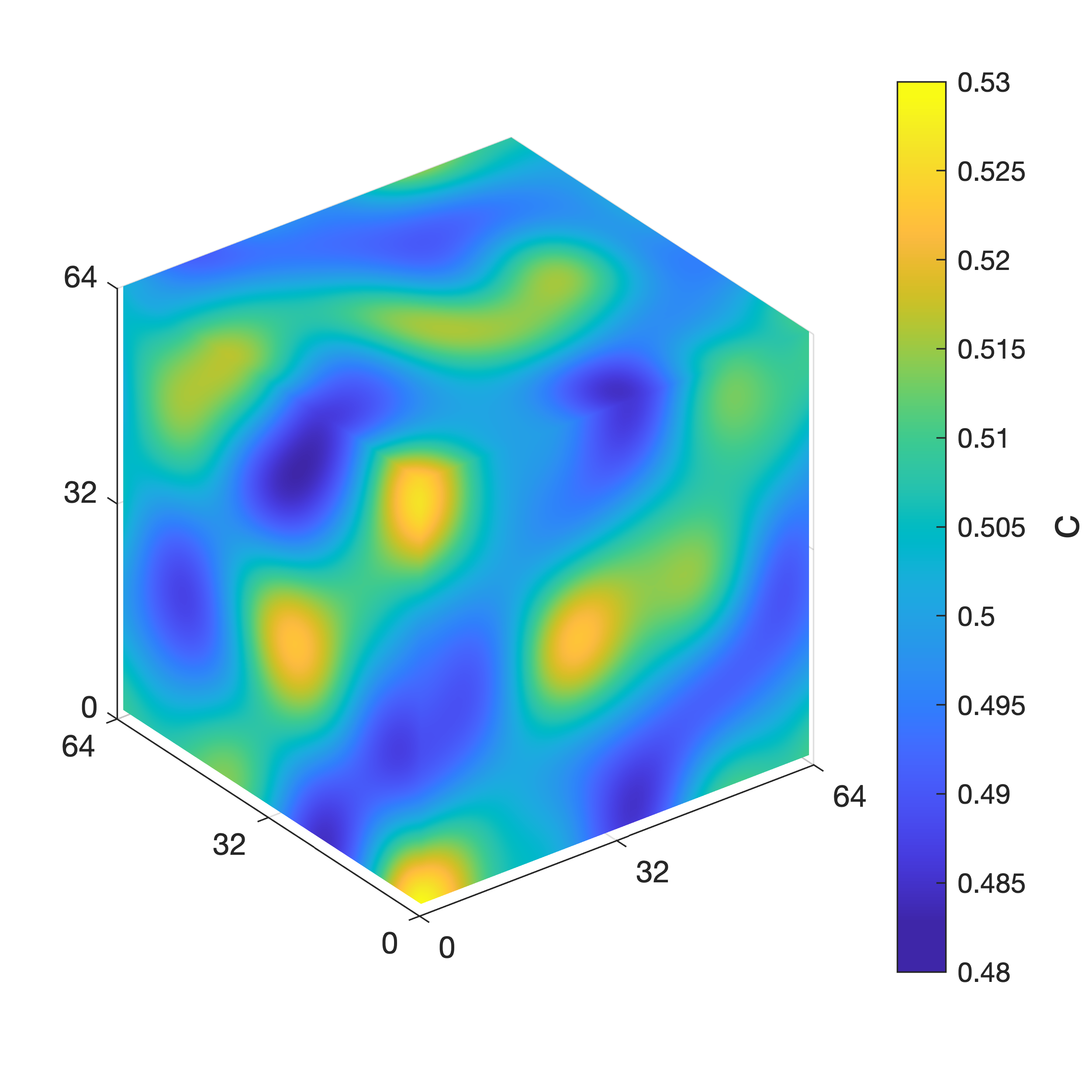}
\caption{3D: $64\times64\times64$}
\end{subfigure}
\caption{The computational domains and initial conditions for the spinodal decomposition benchmark problem in (a) 2D, and (b) 3D.} \label{fig:cahninit}
\end{figure}

The computational domains and initial conditions of the two-dimensional and three-dimensional problems are shown in Fig. \ref{fig:cahninit}.

\subsection{Micro-structural and Free Energy Evolution}

\begin{figure}
\begin{subfigure}{0.49\textwidth}
\includegraphics[width=\linewidth]{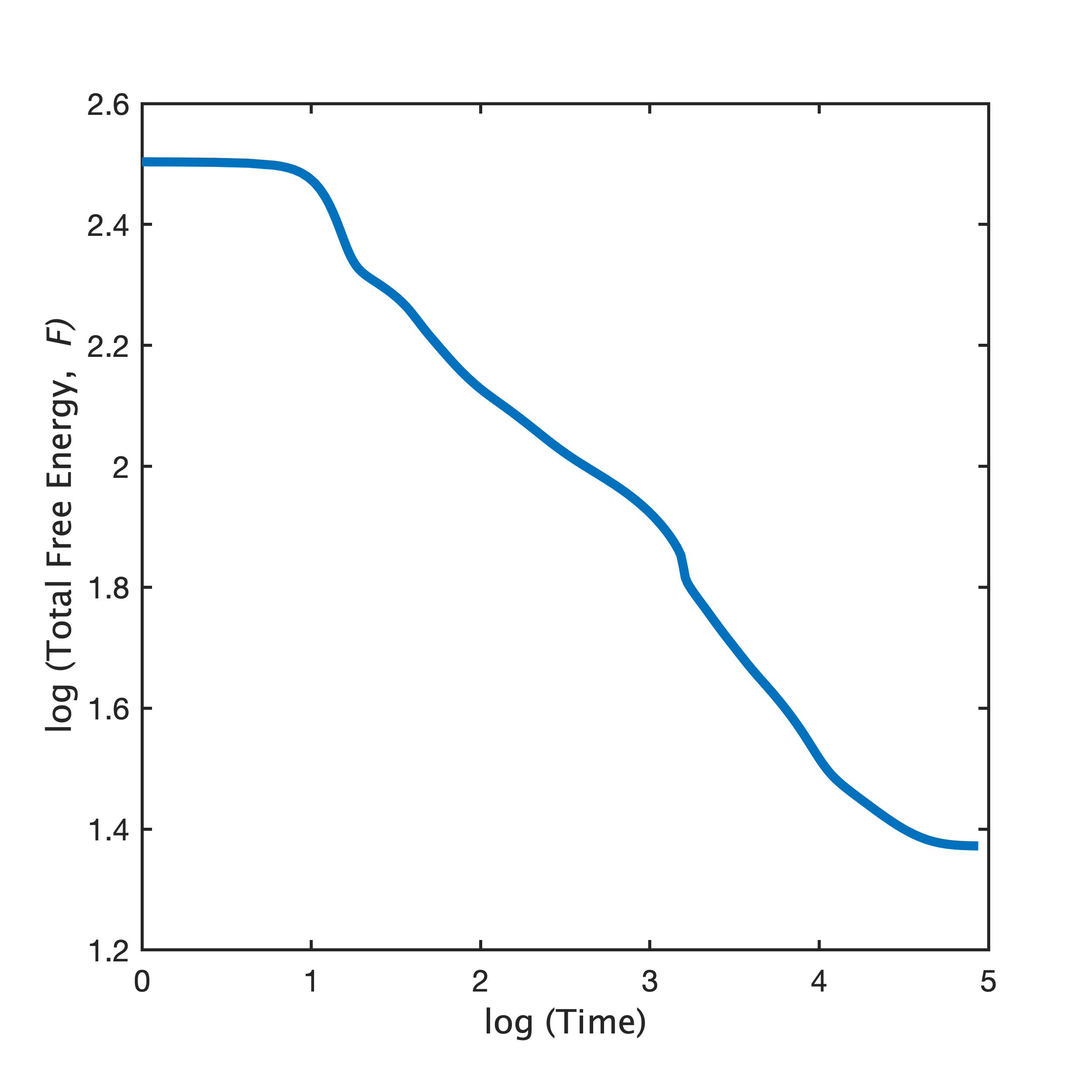}
\end{subfigure}\hspace*{\fill}
\begin{subfigure}{0.49\textwidth}
\includegraphics[width=\linewidth]{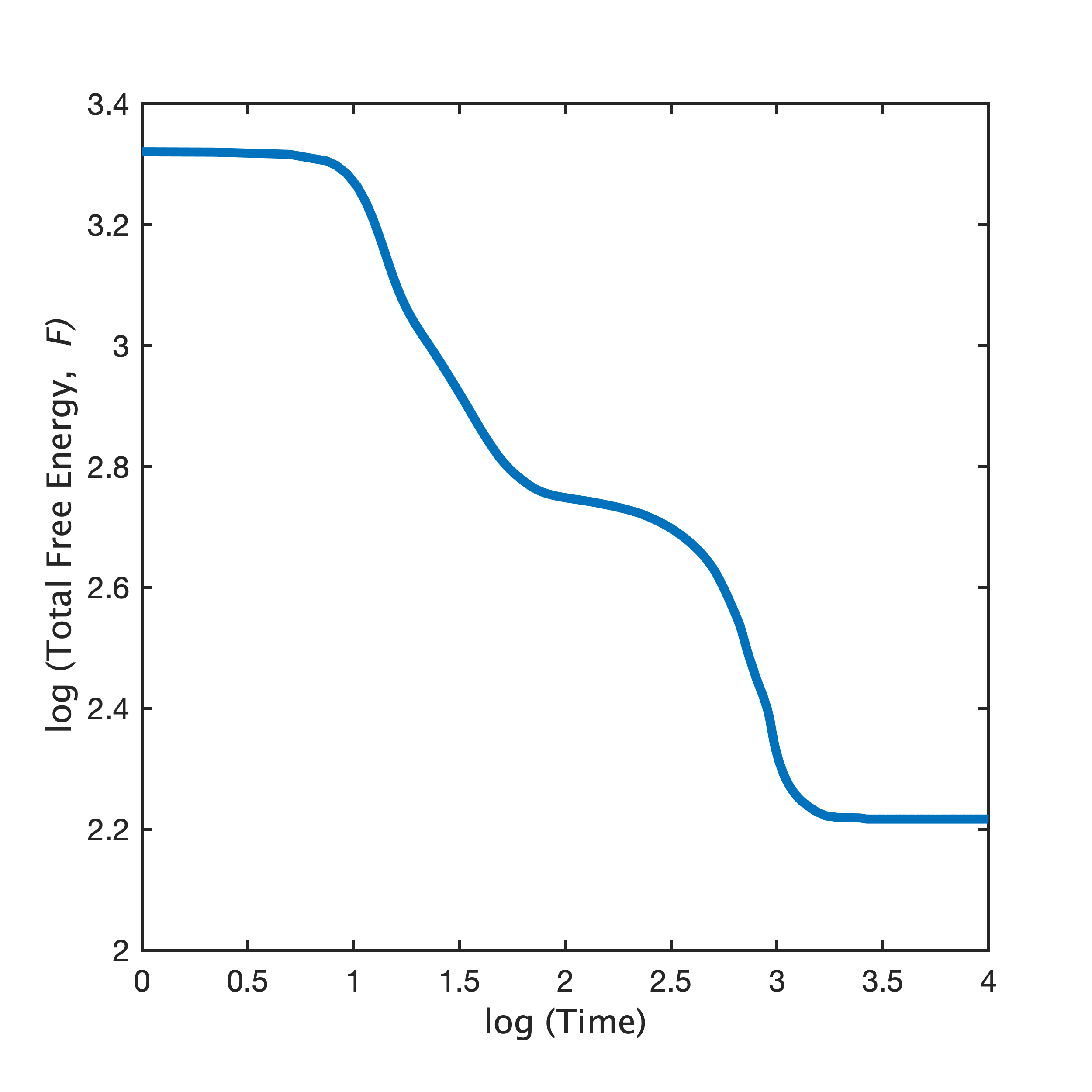}
\end{subfigure}
\caption{The total free energy evolution of the spinodal decomposition benchmark problem in 2D (left), and 3D (right).} \label{fig:cahnF}
\end{figure}

For this benchmark problem, the total free energy of the system and microstructural snapshots are chosen as the metrics to analyze the simulation results. Figure~\ref{fig:cahnF} shows the total free energy evolution of the spinodal decomposition problem. The total free energy decreases rapidly and eventually asymptotically approaches the local energy minimum of the system.

\begin{figure}
\begin{subfigure}{0.49\textwidth}
\includegraphics[width=\linewidth]{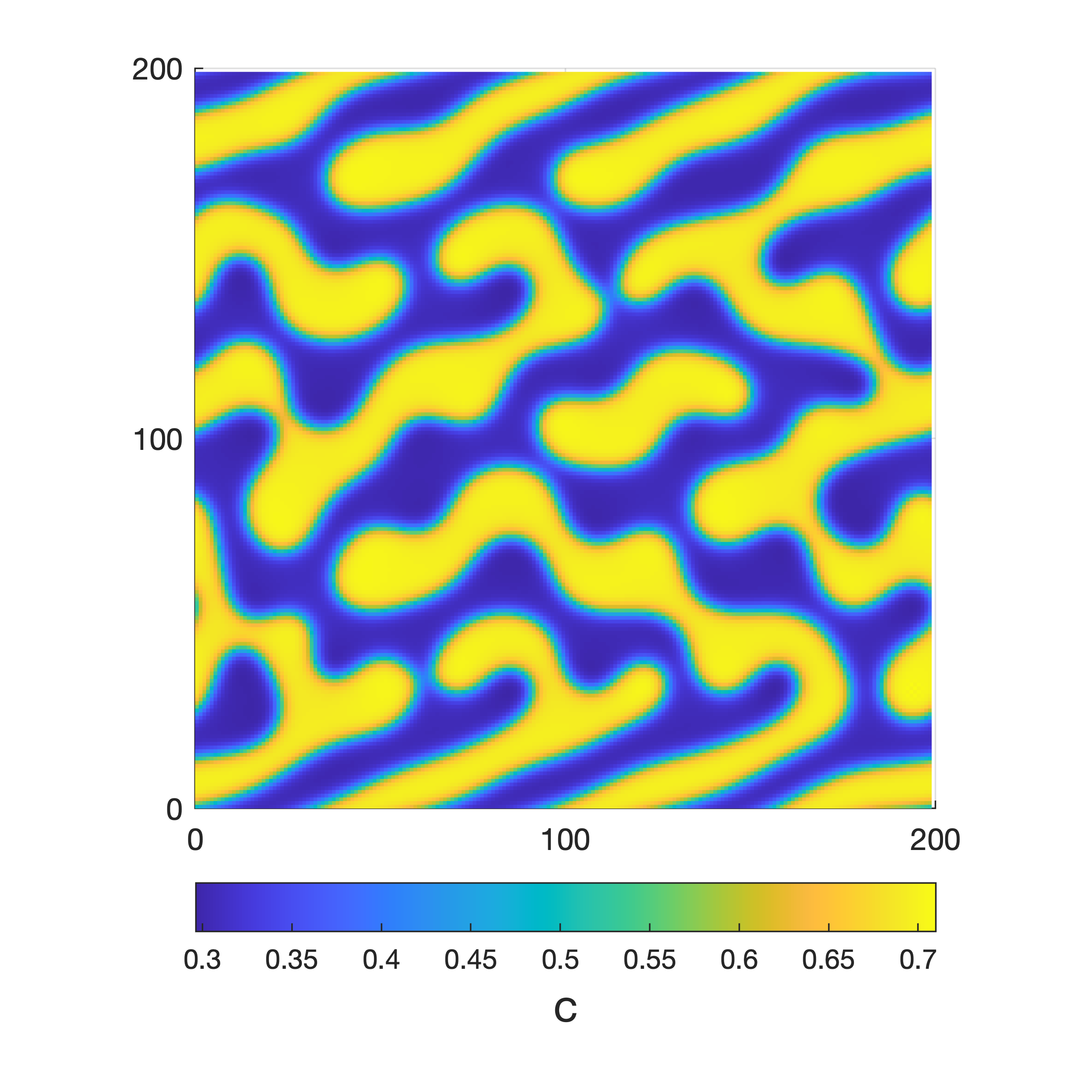}
\caption{$t = 10^2$}
\end{subfigure}\hspace*{\fill}
\begin{subfigure}{0.49\textwidth}
\includegraphics[width=\linewidth]{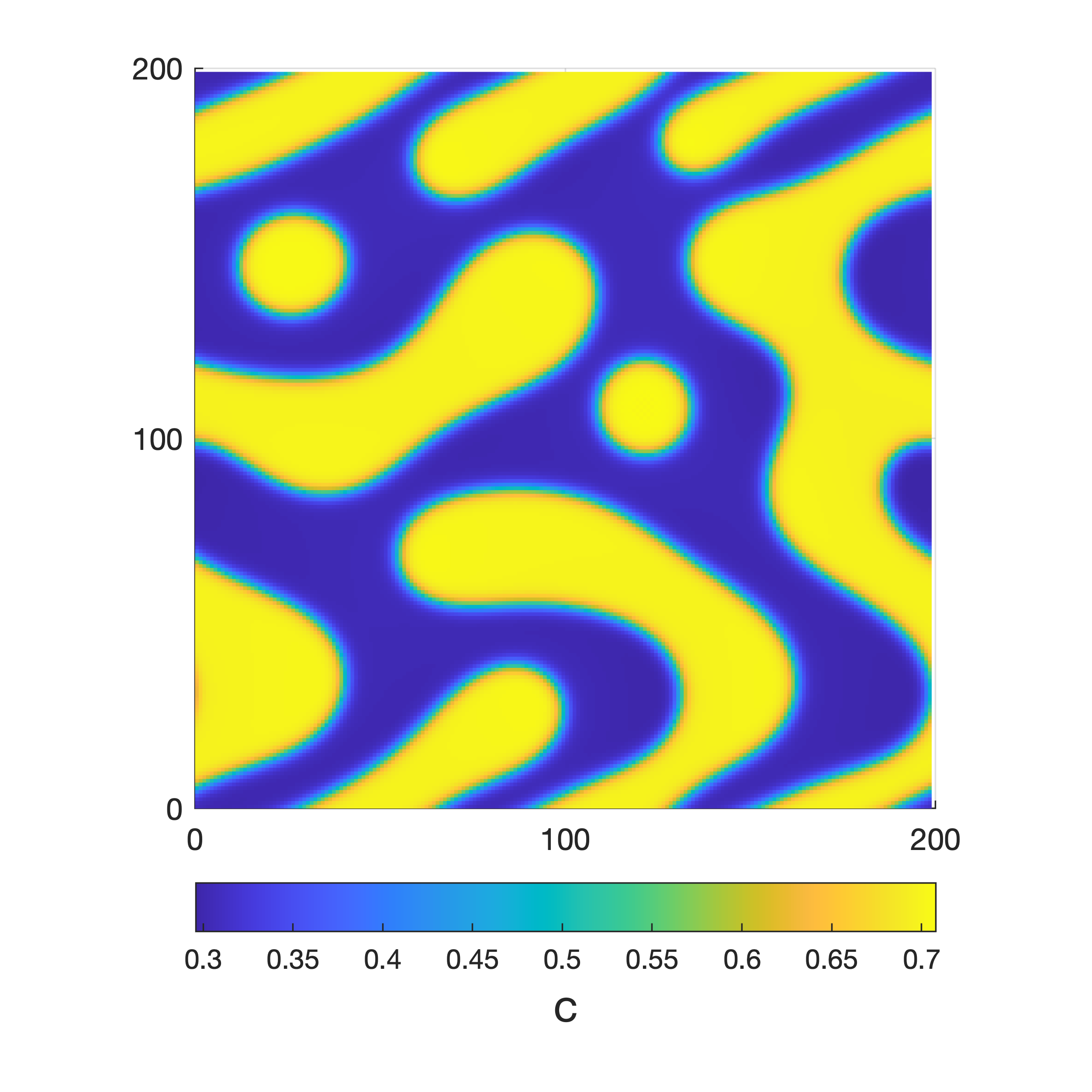}
\caption{$t = 10^3$}
\end{subfigure}
\begin{subfigure}{0.49\textwidth}
\includegraphics[width=\linewidth]{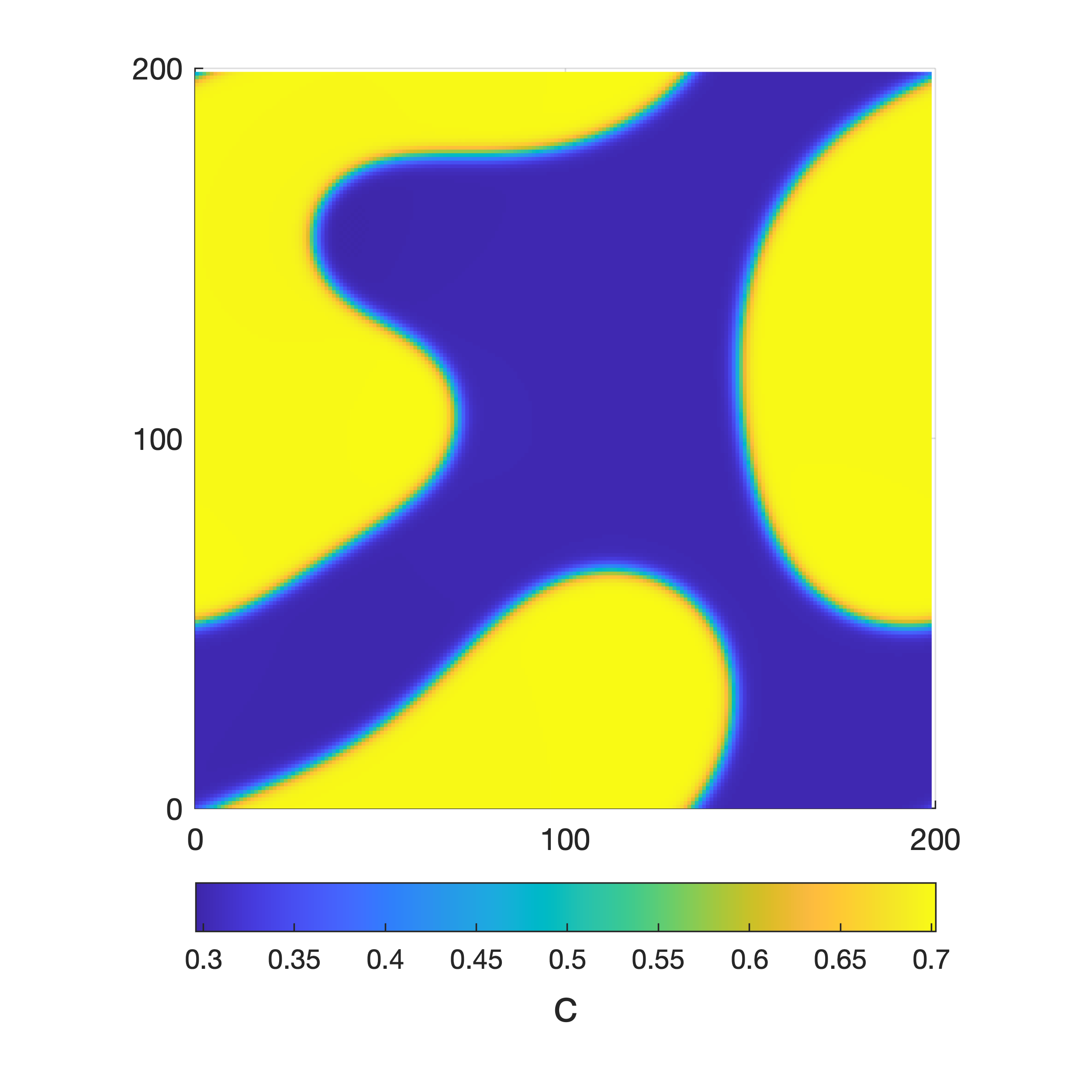}
\caption{$t = 10^4$}
\end{subfigure}\hspace*{\fill}
\begin{subfigure}{0.49\textwidth}
\includegraphics[width=\linewidth]{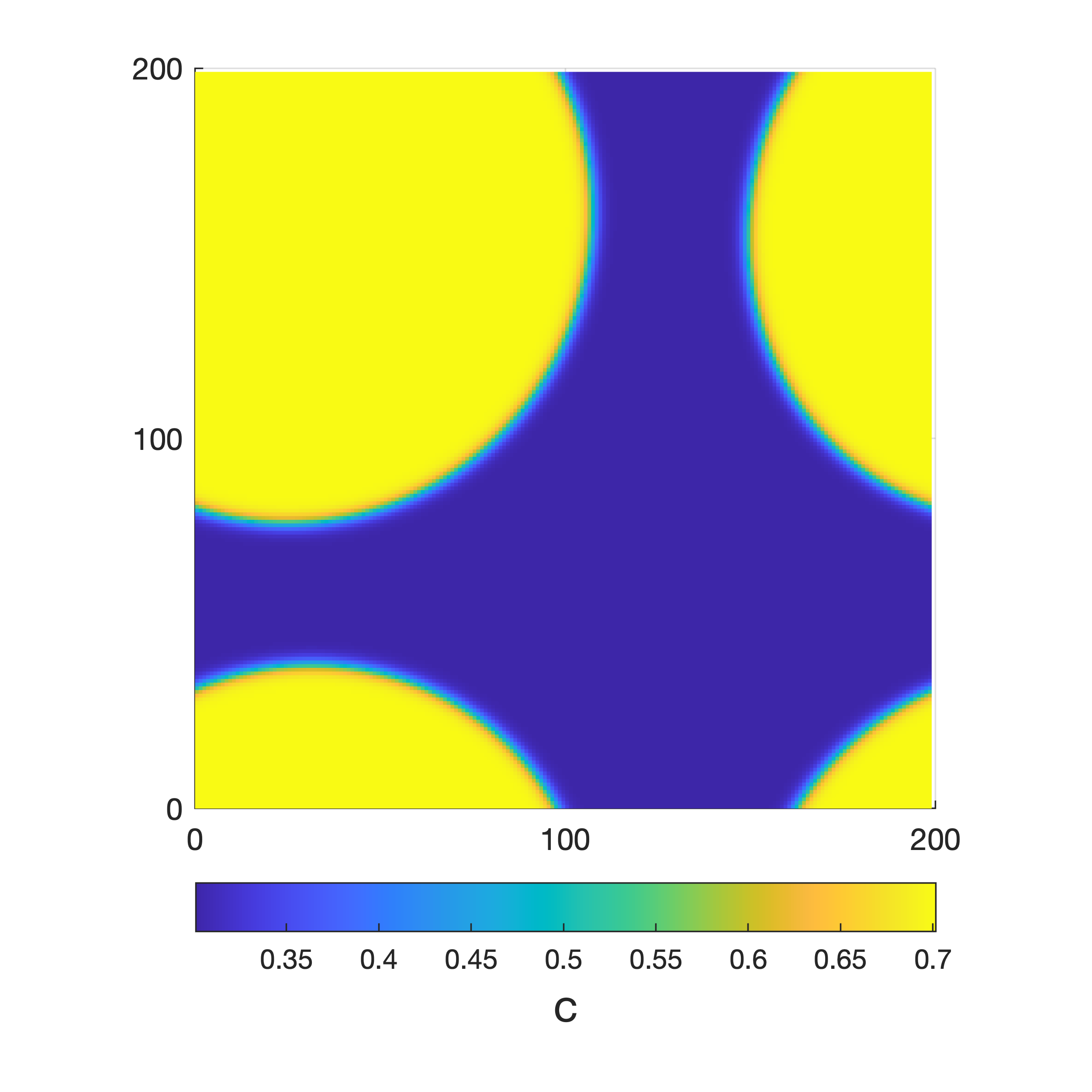}
\caption{$t = 8.7\times10^4$}
\end{subfigure}
\caption{Snapshots of the micro-structure evolution for spinodal decomposition at different time steps in 2D.} \label{fig:cahn2d}
\end{figure}

Figure~\ref{fig:cahn2d} presents the microstructure snapshots for spinodal decomposition of the two-dimensional system at $t = 10^2, 10^3, 10^4$ and $8.7\times10^4$ using at time step of $\Delta t=0.05$. Differences at various times are discernible. Microstructural evolution reaches the lowest energy state beyond $t = 10^5$, but it can be clearly seen from the snapshots that the structure is approaching equilibrium.

\begin{figure}
\begin{subfigure}{0.49\textwidth}
\includegraphics[width=\linewidth]{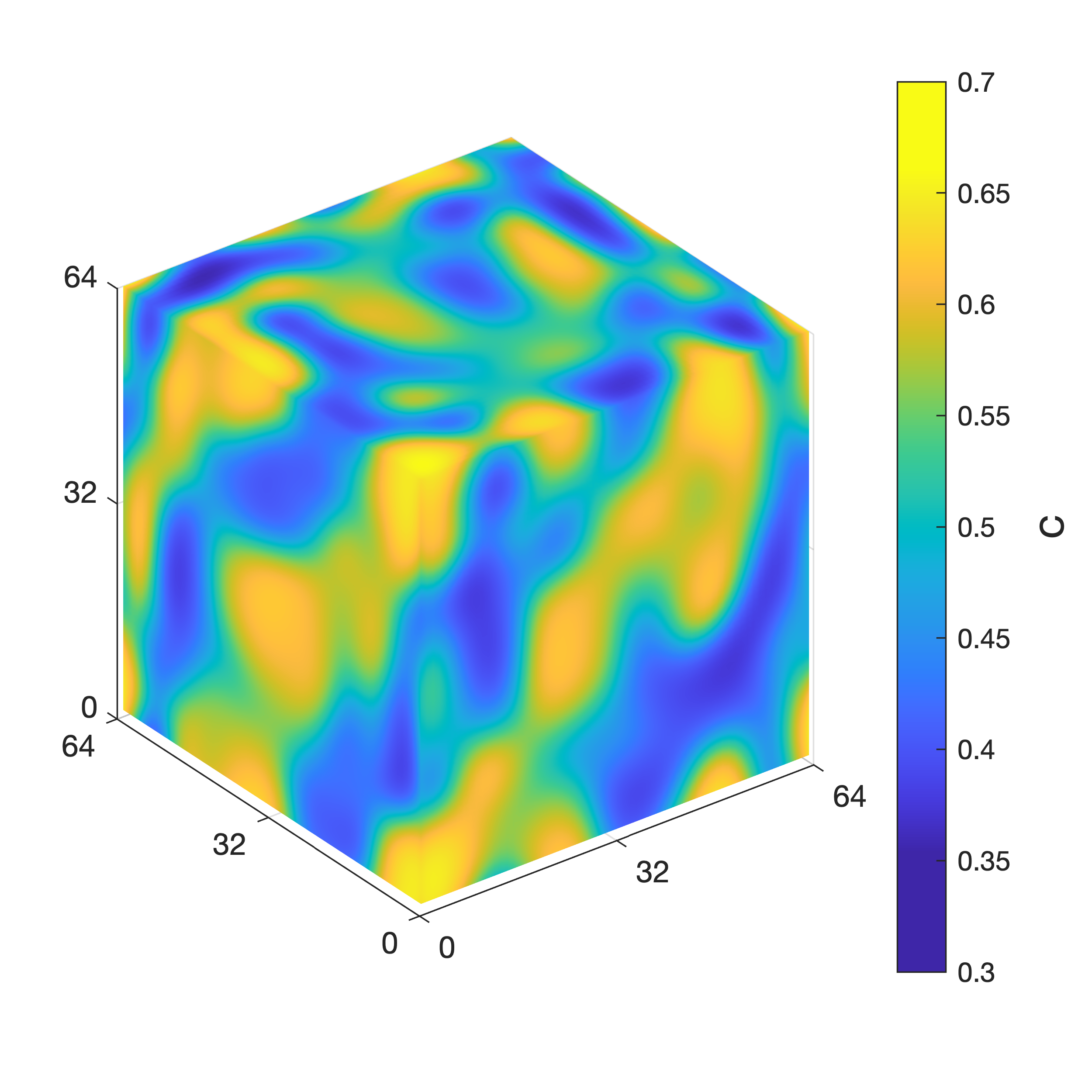}
\caption{$t = 10$}
\end{subfigure}\hspace*{\fill}
\begin{subfigure}{0.49\textwidth}
\includegraphics[width=\linewidth]{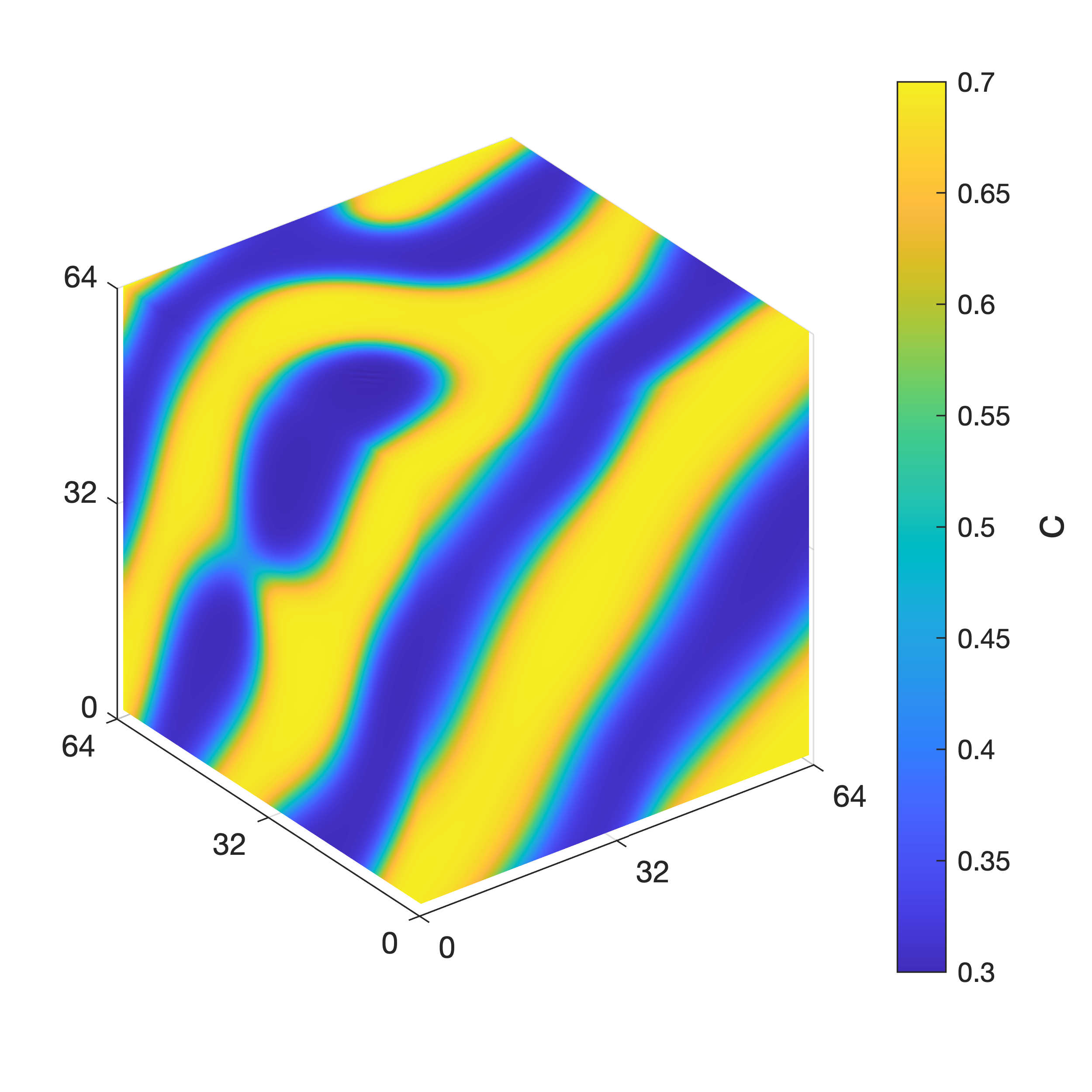}
\caption{$t = 10^2$}
\end{subfigure}
\begin{subfigure}{0.49\textwidth}
\includegraphics[width=\linewidth]{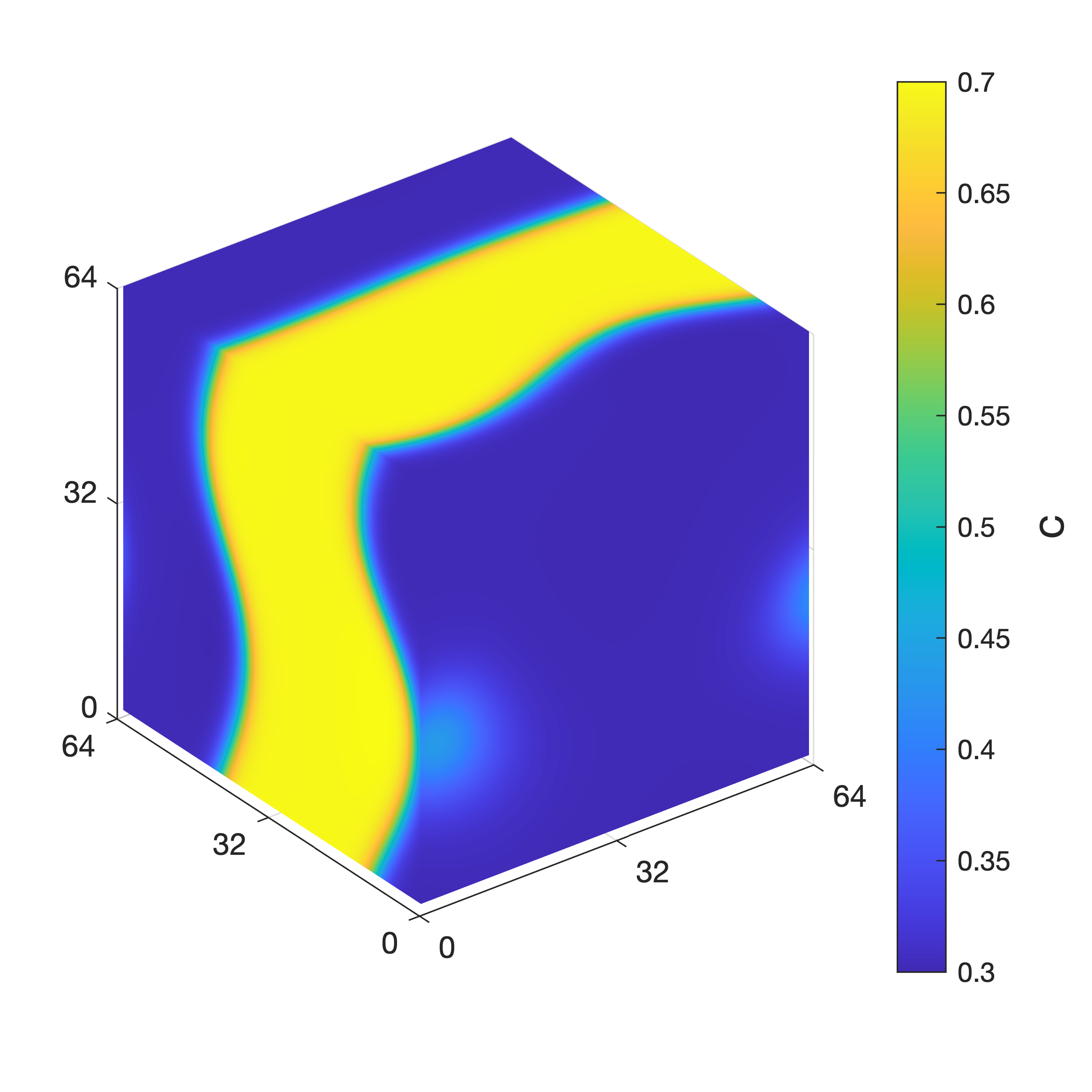}
\caption{$t = 10^3$}
\end{subfigure}\hspace*{\fill}
\begin{subfigure}{0.49\textwidth}
\includegraphics[width=\linewidth]{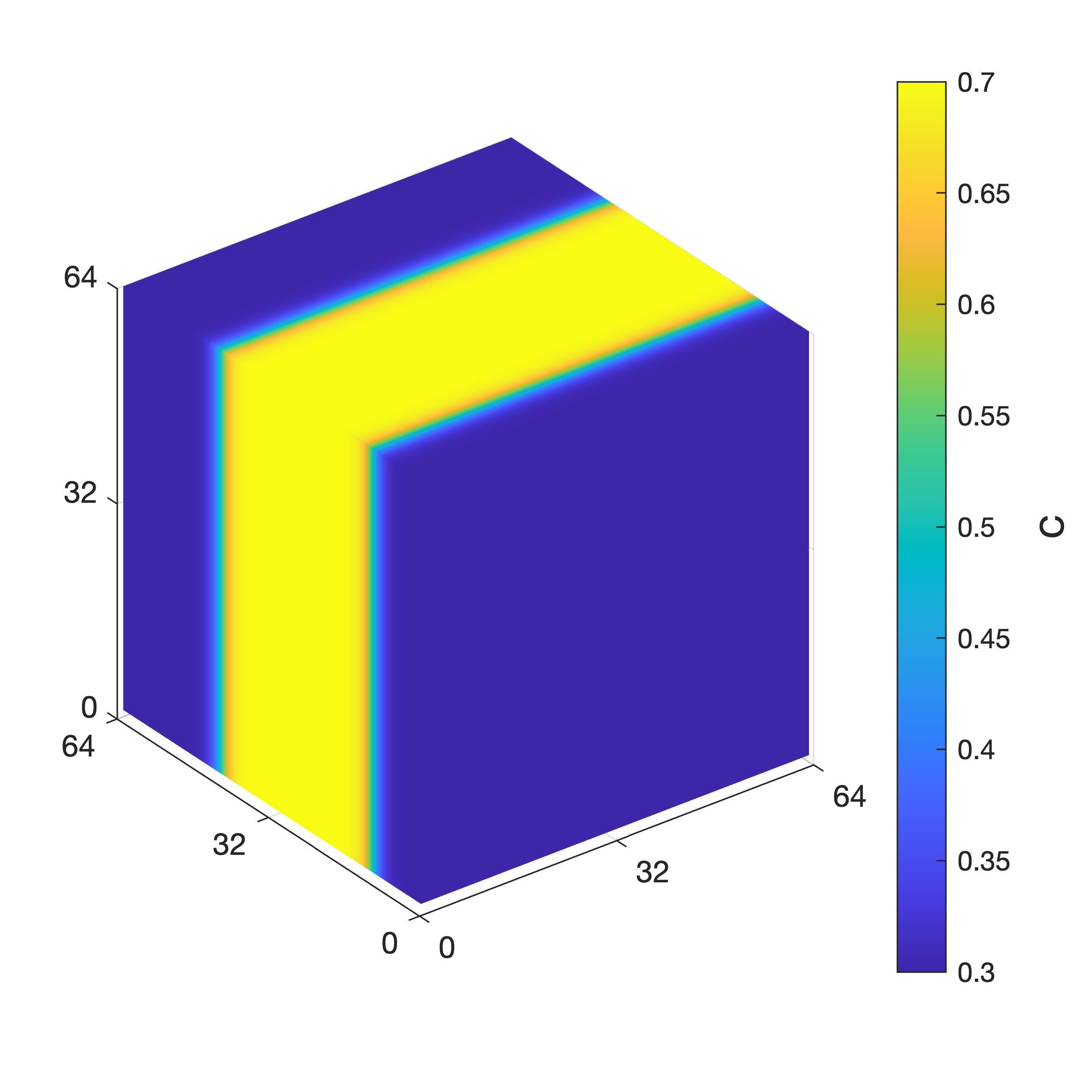}
\caption{$t = 10^4$}
\end{subfigure}
\caption{Snapshots of the micro-structure evolution for spinodal decomposition at different time steps in 3D.} \label{fig:cahn3d}
\end{figure}

Figure \ref{fig:cahn3d} presents the microstructure snapshots for spinodal decomposition at $t = 10$, $10^2$, $10^3$ and $10^4$ for the three-dimensional system using a time step of $\Delta t=0.25$. Differences at various times are quite discernible, heading towards equilibrium at a faster rate. Microstructural evolution reaches the equilibrium somewhere around $t = 5000$, when the total free energy reaches the lowest energy state (seen in Fig. \ref{fig:cahnF}).

\subsection{Temporal Convergence}
Convergence analysis has been performed for the all the prior examples mentioned in the last section. But those examples were only space dependent problems and hence that type of convergence is classified as spatial convergence. For time and space dependent problem, such as this one, the estimation of error in time and space is often independent, and hence a temporal convergence study is done. 

Often times, it is difficult to have an exact solution for time-dependent PDEs. In such cases, one can find the solution for a very low time step, $\Delta t$, and use that as the exact solution to calculate error of the solution obtained from various time steps. Plotting the logarithm of this error against the logarithm of time step, $\Delta t$, would produce a line of slope equal to the temporal convergence rate. Several types of time discretization schemes are available, but for this example, implicit-explicit (IMEX) time discretization scheme has been used.

\begin{figure}
\begin{subfigure}{0.49\textwidth}
\includegraphics[width=\linewidth]{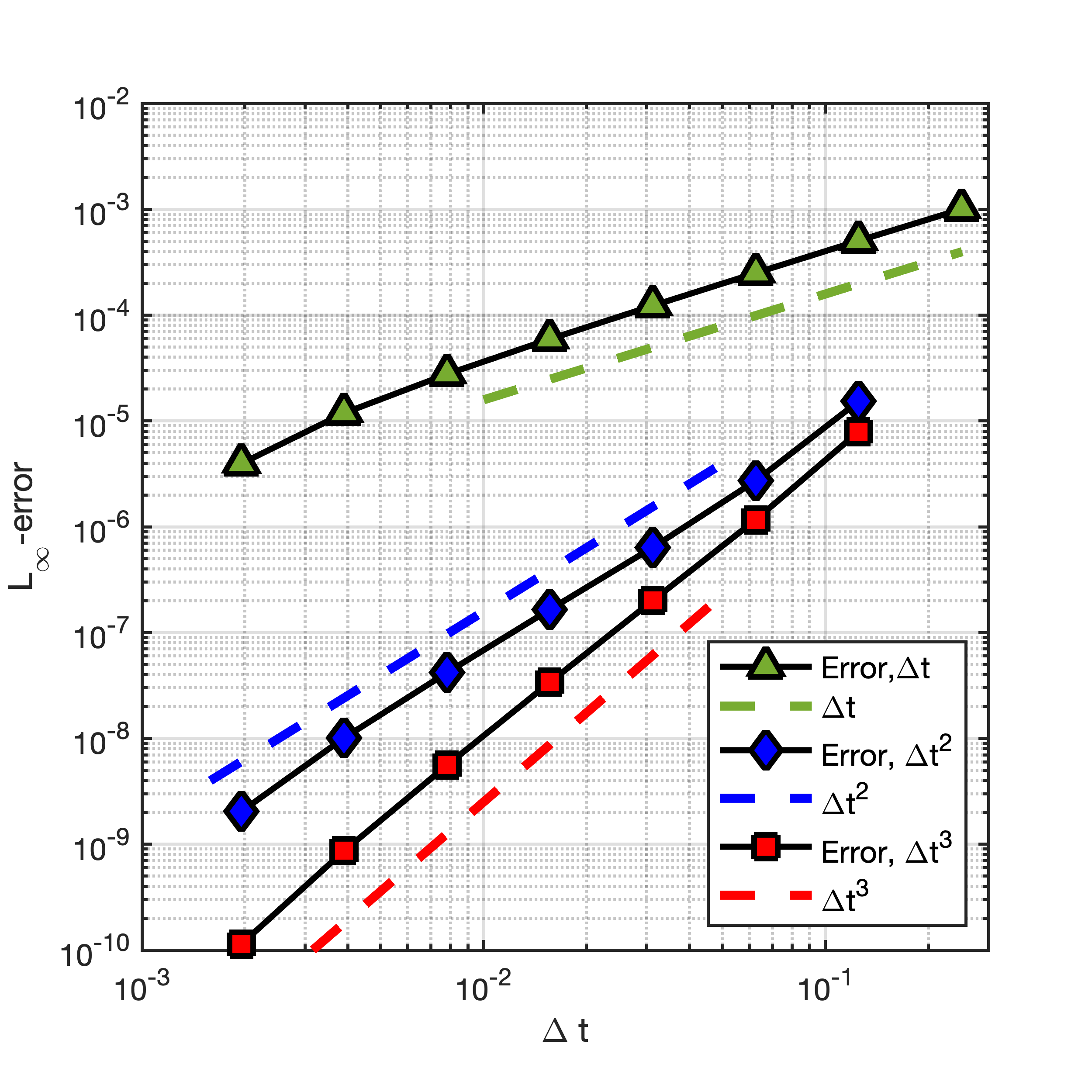}
\caption{2D}
\end{subfigure}\hspace*{\fill}
\begin{subfigure}{0.49\textwidth}
\includegraphics[width=\linewidth]{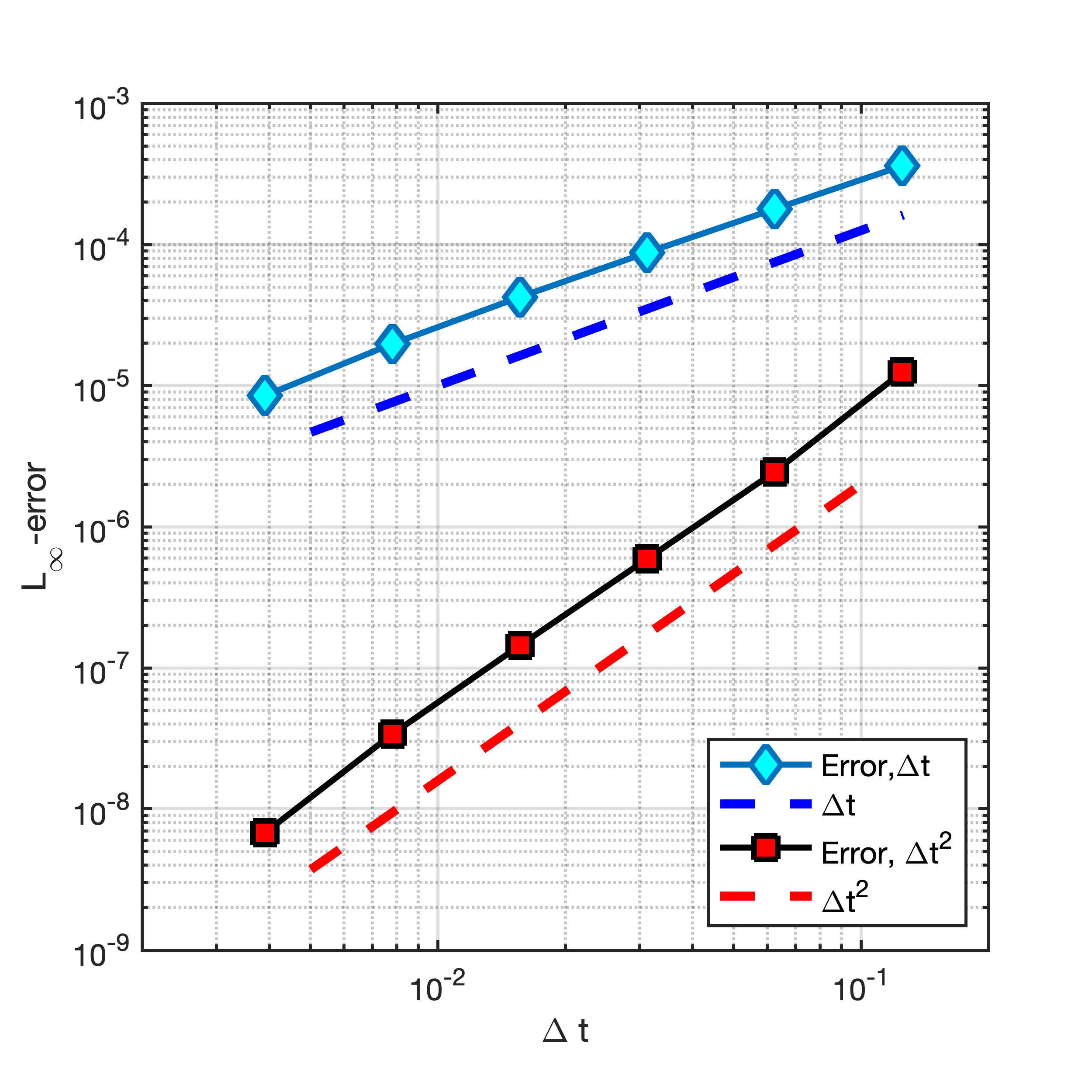}
\caption{3D}
\end{subfigure}
\caption{Temporal Convergence. Comparison of error in the solution when using IMEX time-stepping schemes of $\mathcal{O}(\Delta t)$, $\mathcal{O}(\Delta t^2)$, $\mathcal{O}(\Delta t^3)$ in 2D (left) and $\mathcal{O}(\Delta t)$, $\mathcal{O}(\Delta t^2)$ in 3D (right) for temporal discretization of the Cahn-Hilliard equation .} \label{fig:cahn_conv_dt}
\end{figure}

Figure \ref{fig:cahn_conv_dt} shows the temporal convergence diagram for the spinodal decomposition problem in 2D as well as 3D, with all results compared to the one calculated using $\Delta t=9.7656\times 10^{-4}$ in 2D and $\Delta t = 1.953125\times 10^{-3}$ in 3D. The problem was run multiple times until a final time of t = 10, each time with a different time step ($\Delta t$), for a fixed spacial resolution at h = 1. IMEX time-stepping schemes of orders one, two and three were used to solve the problem in two dimensions and of orders one and two in three dimensions.

In 2D (Fig. \ref{fig:cahn_conv_dt} (a)), for the scheme with order one, the solution initially converges at the expected rate, but a slight deviation can be observed when the time step is very low. For the schemes orders two and three, the solution can be seen converging at the expected rate. In 3D (Fig. \ref{fig:cahn_conv_dt} (b)), for both the methods, of $\mathcal{O}(\Delta t)$ and $\mathcal{O}(\Delta t^2)$, the solution converges the an expected rate with no deviations or tailing off at any point.

\subsection{Static-Scaling}
Traditionally, there has always been two kinds of scaling analysis, strong-scaling and weak-scaling. But from both strong-scaling or weak-scaling plots, it is unclear how a given machine or algorithm will handle a variety of workloads~\cite{tas}. This leads to concept of \emph{static-scaling}~\cite{weakscaling, tas}, where the problem size is increased for a fixed parallelism. For a detailed explanation on static-scaling, please refer to~\cite{staticscaling}. 

In this type of scaling analysis, the time to solution is plotted against the total degrees of freedom (DoF) solved per second for a variety of problem sizes. Optimal scaling will be indicated by a horizontal line as the problem size is increased, which is the middle region of the plot. On the left, there might sometimes be delay in reaching optimal scaling which may be due to the fact that problem size is too small for a given number of MPI processes, and hence the large communication to computation ratios (strong-scaling effects). Sometimes, one might also see some tailing off to the right of the static-scaling plot. This may be caused by how the memory is allocated and accessed. Larger problem sizes may see an increase in time to access the main memory. The static-scaling plots are therefore designed to capture both strong-scaling and weak-scaling characteristics on the same plot, and is therefore a very good indicator of the ideal range of problem sizes for a given number of MPI processes~\cite{staticscaling}.

\begin{figure}
    \centering
    \includegraphics[width=\textwidth]{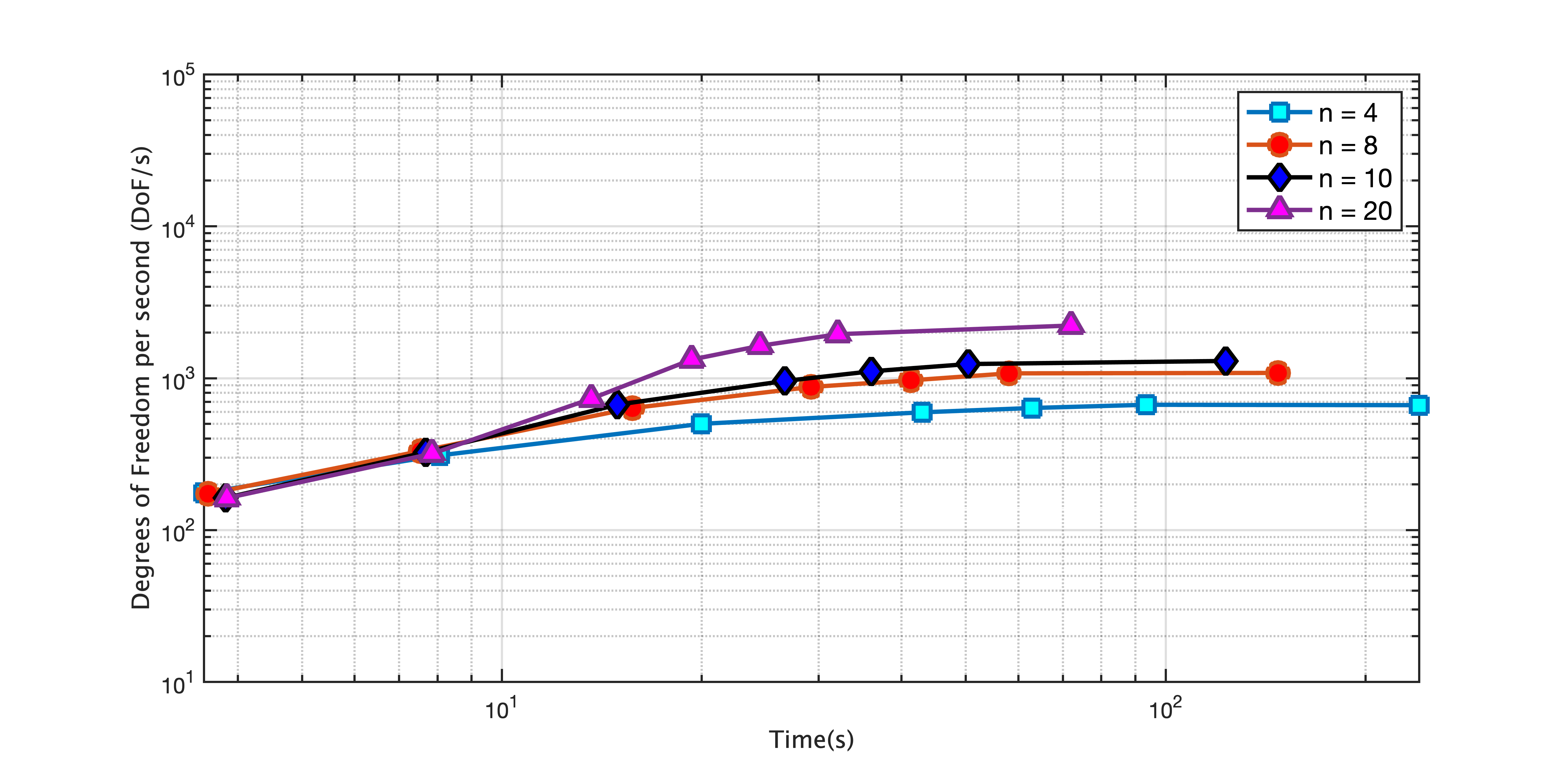}
    \caption{Static-scaling behavior of the spinodal decomposition problem using Cahn-Hilliard equation in 2D. Performance is compared for varying problem sizes running on different number of MPI processes ($n$).}
    \label{fig:cahn_scaling_2d}
\end{figure}

Figure \ref{fig:cahn_scaling_2d} shows the static-scaling plot of the spinodal decomposition problem in 2D. Problem sizes range from 625 degrees of freedom to $1.6\times10^5$. The problem was run using 4, 8, 10 and 20 MPI processes ($n$). The scaling plot indicates that method has excellent weak-scaling as the degrees of freedom solved per second remains constant as the problem size increases, given by the solution time. There are no memory effects as the problems get larger. The solution also shows good strong-scaling behaviour, as the performance improves with increasing number of MPI processes. The strong-scaling limit is a little different here. The time is takes to solve a problem with smaller sizes is very similar across all processes. Regardless, $n=4$ still has the least amount of overhead, and gives an almost-flat line throughout, while for $n = 20$, the static-scaling limit is the most apparent.

\begin{figure}
    \centering
    \includegraphics[width=\textwidth]{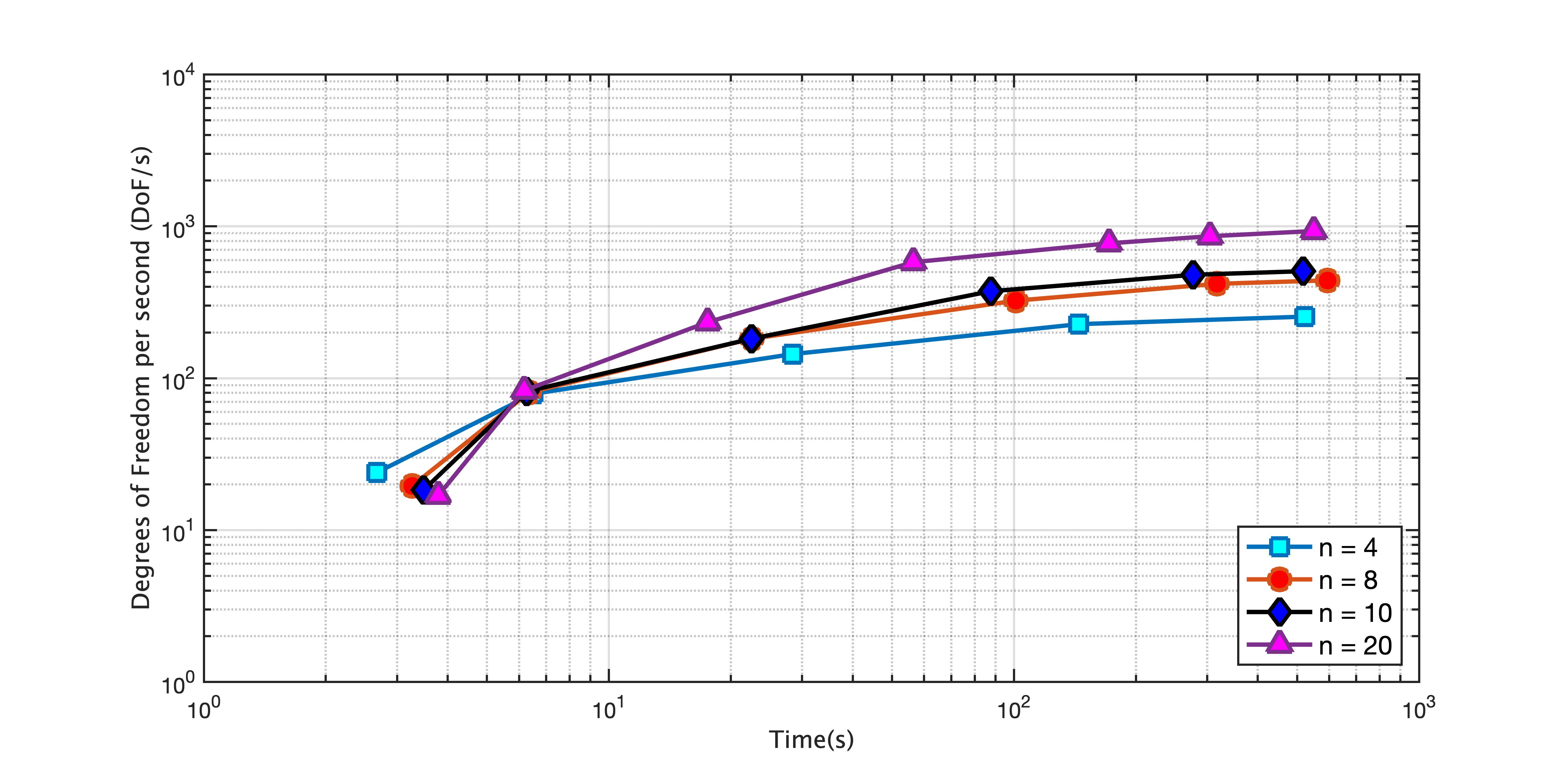}
    \caption{Static-scaling behavior of the spinodal decomposition problem using Cahn-Hilliard equation in 3D. Performance is compared for varying problem sizes running on different number of MPI processes ($n$).}
    \label{fig:cahn_scaling_3d}
\end{figure}

Figure \ref{fig:cahn_scaling_3d} shows the static-scaling plot of the spinodal decomposition problem in 3D. Problem sizes range from 64 degrees of freedom to $2.6\times10^5$. Similar to 2D, the problem was run using 4, 8, 10 and 20 MPI processes ($n$). The scaling plot shows a slight different behaviour for this problem. Even though the solution is strong-scaling well, since $n=20$ is outperforming other MPI processes, the problem is yet to achieve optimal scaling as no flat line can be seen in the middle region. Running the problem with increased problem sizes would show better weak-scaling results. However, there are no memory effects as the problems get larger. The strong-scaling limit has the expected behavior with $n=4$ having the least amount of overhead, and $n = 20$ the most.

\subsection{Eigenvalues, spectral radius and condition number}

\begin{figure}
    \centering
    \includegraphics[width=0.75\textwidth]{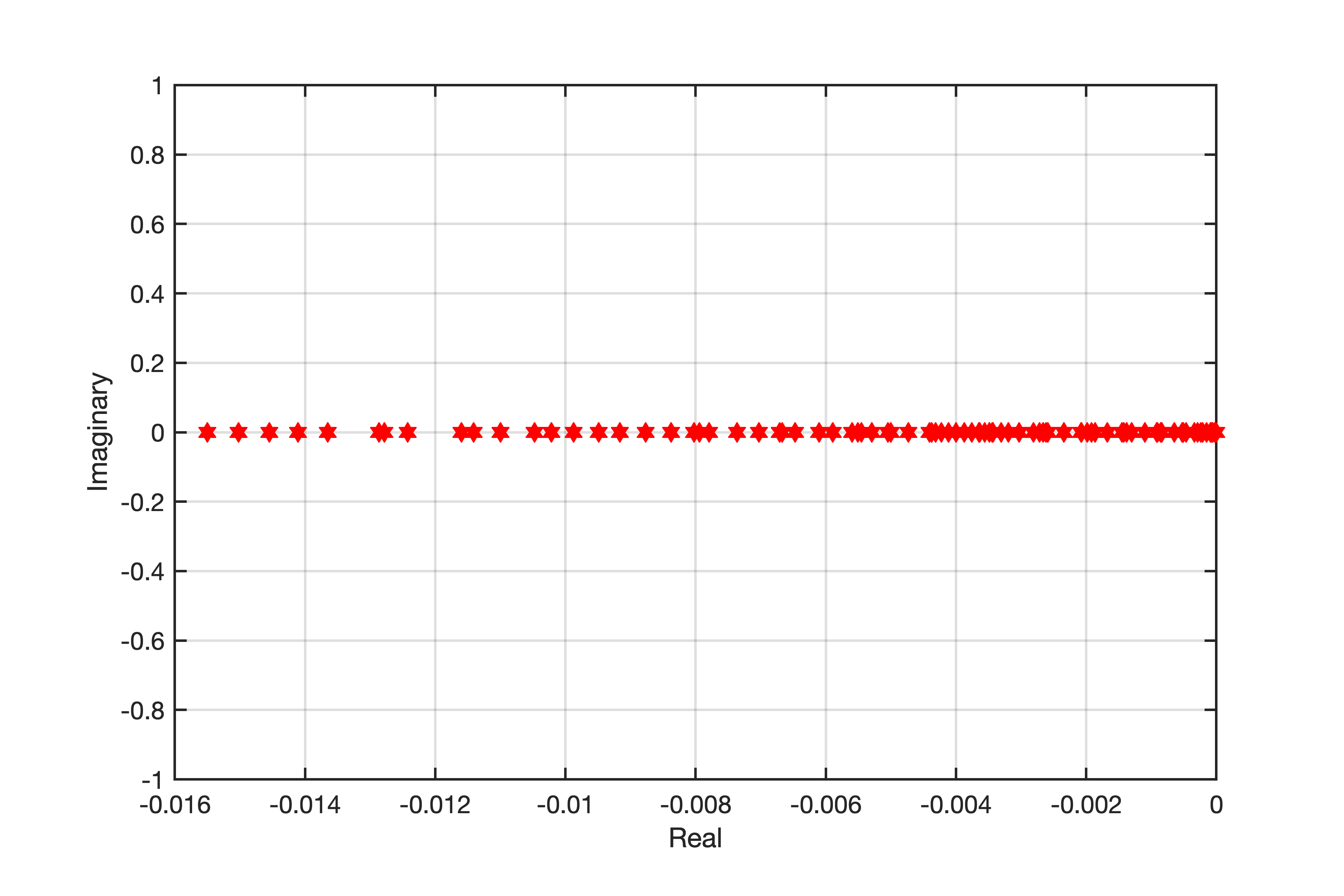}
    \caption{The spectrum of the bi-Laplacian matrix used in the spinodal decomposition problem.}
    \label{fig:spectrum}
\end{figure}
Next, consider the properties of the bi-Laplacian matrix used to solve the spinodal decomposition problem shown earlier.
Figure \ref{fig:spectrum} shows the spectrum, i.e., the distribution of the eigenvalues in the complex plane, of the 2D system. From the plot, it can be observed that the eigenvalues of the matrix only have real negative parts, hence leading to stability.

\renewcommand{\arraystretch}{1.25}
\begin{table}
\centering
\caption{Properties of the bi-Laplacian matrix (B: $\Delta(\Delta c)$) and the time-stepping matrix used for Cahn-Hilliard equation in 2D for a domain size of $[0,200]^2$ \textit{($\rho$: spectral radius)}.}
\begin{tabular}{ C{2cm} | C{2cm} | C{2cm} C{2cm} C{2.5cm} }
\toprule
$h$ & $\Delta t$ & $\rho(B)$ & $\rho(I+\Delta t B)$ & cond ($I+\Delta t B$) \\
\midrule
 8 & 2 & 0.0155 & 1.031 & 1.031\\
 4 & 1 & 0.25 & 1.25 & 1.25\\  
 2 & 0.5 & 4 & 3 & 3\\
 1 & 0.25 & 64 & 17 & 17\\
 0.5 & 0.125 & 1024 & 129 & 129\\
\bottomrule
\end{tabular}
\label{tab:bilap_prop}
\end{table}

Table \ref{tab:bilap_prop} lists the spectral radius, i.e., the absolute value of the largest eigenvalue, of the bi-Laplacian matrix and the corresponding time-stepping matrix, along with its condition number, for various spatial resolutions with $\Delta t = \frac{1}{4} h$. As the time-stepping matrix is normal, the condition number is given by the (absolute) ratio of the largest eigenvalue to the smallest one. As the smallest eigenvalue is almost one in all these cases, the condition number stays approximately equal to the spectral radius of the time-stepping matrices. Note that in the Cahn-Hilliard equation \eqref{eq:cahn}, the bi-Laplacian matrix on the right hand side of the equation has a negative sign in front of it. Therefore, here the eigenvalues correspond to negative of the bi-Laplacian matrix ($-B$). The time-stepping matrix thus becomes $[I+\Delta t B]$, which is what is used to check the condition number. It should come as no surprise that the spectral radius increases as the resolution decreases, as the matrix is getting bigger. The condition numbers are relatively small when the matrices are small and increase as resolution decreases, but the overall matrices are still well conditioned.

\subsection{Sparsity}

\begin{table}
\centering
\caption{Sparsity properties of Laplacian matrix and bi-Laplacian matrix used for the spinodal decomposition problem.}
\begin{tabular}{ C{0.5cm} | C{3cm} | C{1.5cm} C{1.8cm} | C{1.5cm} C{1.8cm}}
\toprule
\multicolumn{2}{C{3.5cm}}{} & \multicolumn{2}{C{3.3cm}}{Laplacian} & \multicolumn{2}{ C{3.3cm}}{bi-Laplacian}\\
\midrule
$h$ & size & non-zero entries & percentage & non-zero entries & percentage \\
\midrule
 8 & $625\times625$  & 3125 & 0.8 & 8125 & 2.08  \\ 
 4 & $2500\times2500$ & 12500 & 0.2 & 32500 & 0.52\\  
 2 & $10,000\times10,000$ & 50,000 & 0.05 & 130,000 & 0.13\\
 1 & $40,000\times40,000$ & 200,000 & 0.0125 & 520,000 & 0.0325\\
 0.5 & $160,000\times160,000$ & 800,000 & 0.003125 & 2,080,000 & 0.008125 \\
 \bottomrule
\end{tabular}
\label{tab:sparsity}
\end{table}

Table \ref{tab:sparsity} shows the sparsity properties of the bi-Laplacian matrix. The matrix is getting sparser by $75\%$ each time the resolution is reduced by half. Recall that bi-Laplacian matrix is obtained using \scomp of two Laplacian matrices. In fact, Laplacian is a five-point stencil, which gives a thirteen-point stencil when composed (in 2D) with itself:
\begin{align*}
    \text{Laplacian:}
        \left(
        \begin{array}{ccc}
         & 1 &  \\
        1 & -4 & 1 \\
         & 1 &   \\
        \end{array}
        \right),
    \quad\quad
    \text{bi-Laplacian:}
        \left(
        \begin{array}{ccccc}
         & & 1 &  & \\
         & 2 & -8 & 2 &\\
        1&-8 & 20 & -8 & 1\\
        & 2 & -8 & 2 &\\
        & & 1 &  & \\
        \end{array}
        \right).
\end{align*}

\begin{figure}
\begin{subfigure}{0.49\textwidth}
\includegraphics[width=\linewidth]{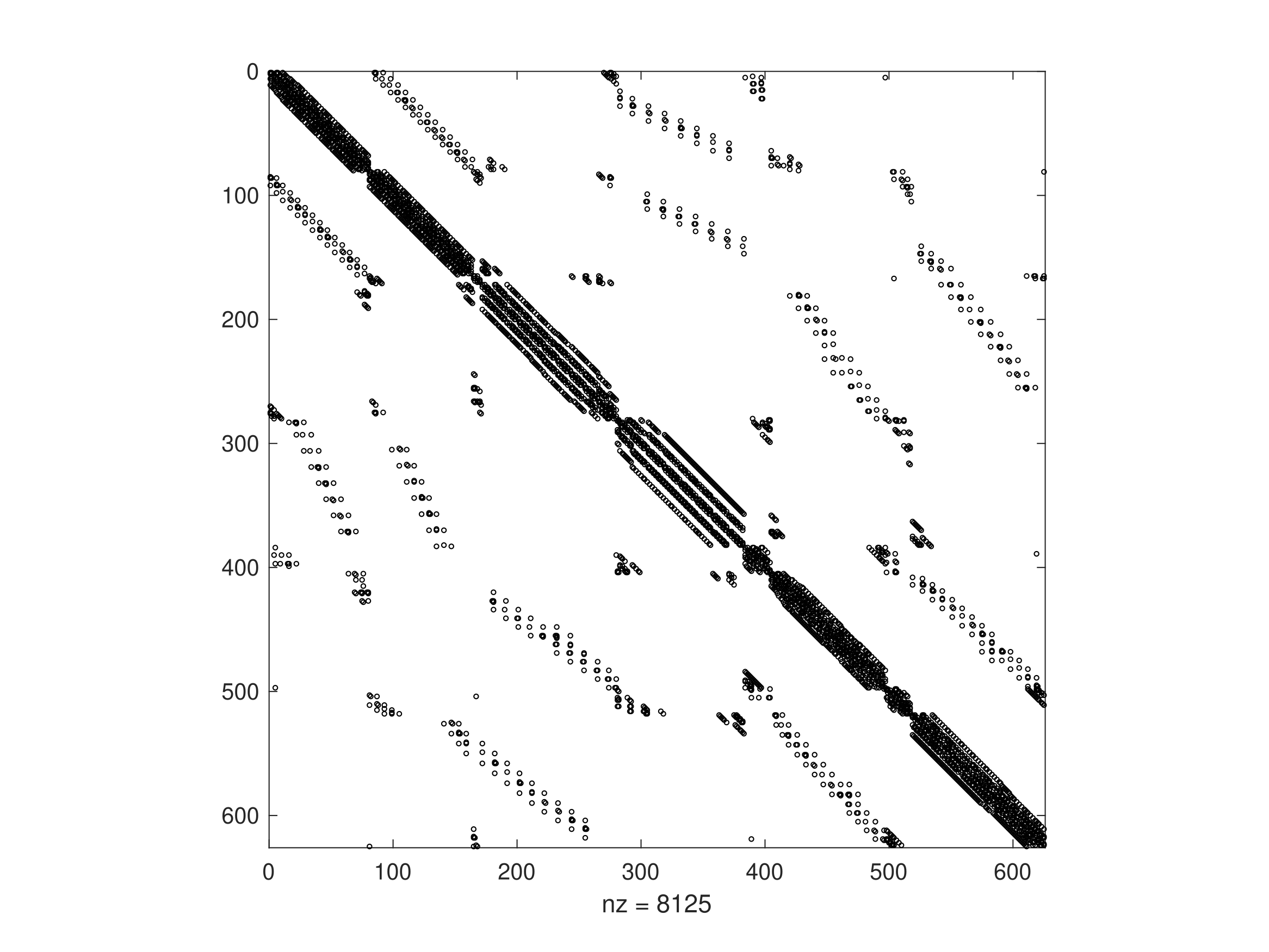}
\caption{$h = 8$}
\end{subfigure}\hspace*{\fill}
\begin{subfigure}{0.49\textwidth}
\includegraphics[width=\linewidth]{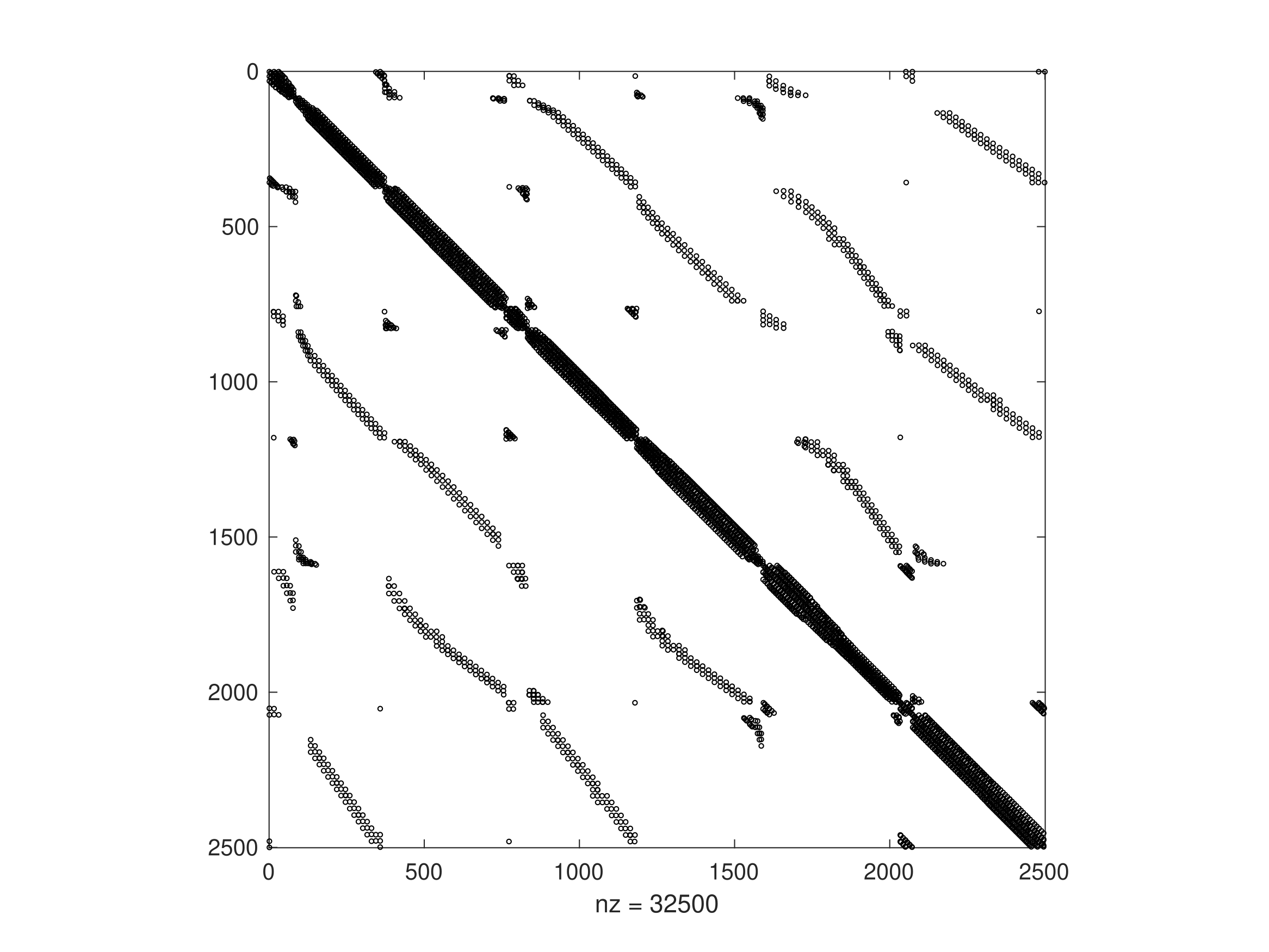}
\caption{$h = 4$}
\end{subfigure}
\begin{subfigure}{0.49\textwidth}
\includegraphics[width=\linewidth]{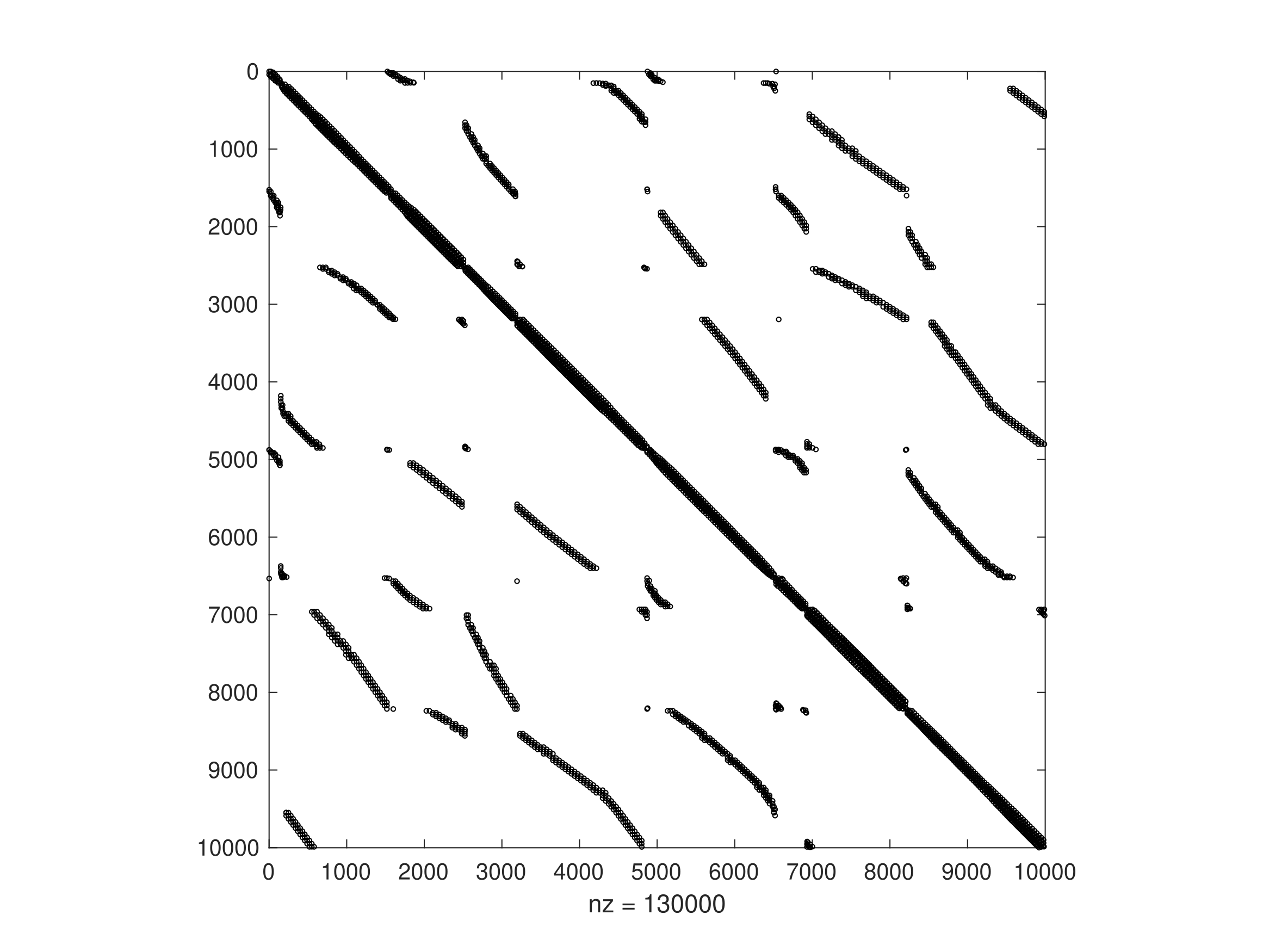}
\caption{$h = 2$}
\end{subfigure}\hspace*{\fill}
\begin{subfigure}{0.49\textwidth}
\includegraphics[width=\linewidth]{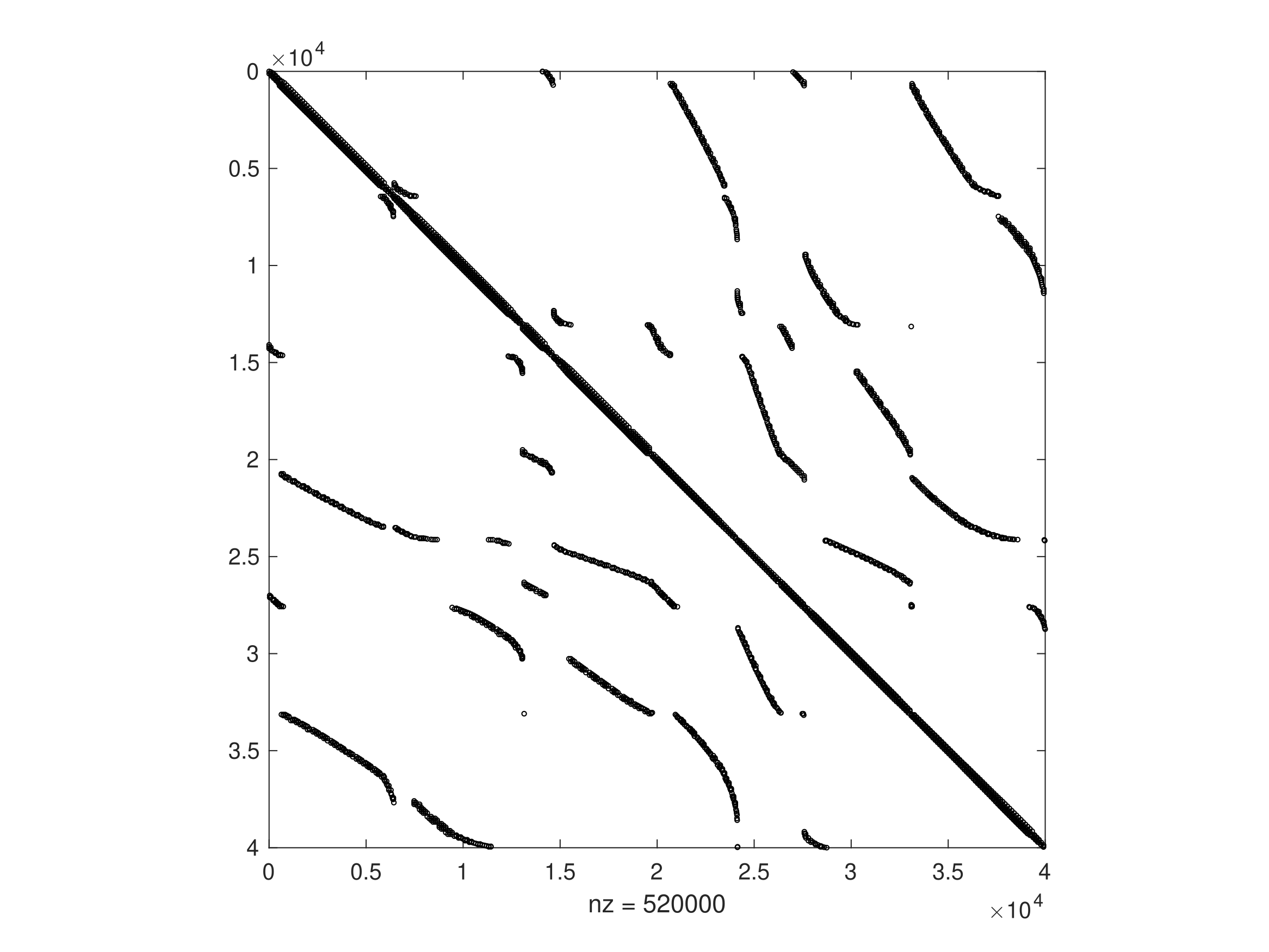}
\caption{$h = 1$}
\end{subfigure}
\caption{Sparsity pattern of the Cahn-Hilliard time-stepping matrices for various $h$.} \label{fig:sparsity}
\end{figure}

Note that the ratio between the non-zero entries of Laplacian and bi-Laplacian matrices stays consistent at at a ratio of $5:13$, thus, validating that the correct sparsity pattern is maintained when stencils are composed.

Figure \ref{fig:sparsity} shows the sparsity structure of bi-Laplacian matrices at different resolutions used for the spinodal decomposition problem in 2D. The figure shows that the same sparsity structure is maintained for various resolutions.

\section{Concluding Remarks}
\label{sec:conc}

In this work we make use of composition to form finite difference stencils which can then be used to numerically evaluate derivatives and solve partial differential equations. In stencil composition, two stencils with arbitrary derivative-orders are composed to obtain a stencil with a derivative-order equal to the sum of each individual stencil approximation. We represent stencils for various orders of derivative as stencil vectors, with the elements being the coefficients of the truncation error terms, allowing for the determination of the leading-order error term of the composed stencil. We show that stencil composition is associative and prove that the order-of-accuracy of the composed stencil will never fall below the lowest-order accuracy of the stencils being composed. The stability of stencil composition is also explored using two example PDEs. Numerical examples, both in one and two-dimensions, verify our findings regarding order of accuracy via convergence tests. A PDE application is also shown, wherein a boundary value problem involving the biharmonic equation is discretized using the composition of two Laplacian stencils, and the convergence rate is verified. A benchmark problem involving the Cahn-Hilliard equation based model was solved using a two-dimensional as well as a three-dimensional computational domain. For this experiment,  performance analyses were conducted which included static-scaling analysis and temporal convergence tests, thus validating the discretization method.

There is an important caveat of this work that must also be discussed. The work here only holds if the inner stencils for a composition do not vary from location-to-location in the outer stencil. For example, consider composing two finite difference stencils of the second-derivative in a domain. At the boundary it may be tempting to mix one-sided stencils with center-stencils. This is not advised, as the composition is no longer between a single inner and single outer function, but different functions, which creates unpredictable results. As a demonstration consider the composition of first-derivative approximations to obtain a second-derivative stencil. Let the outer stencil be given by $\vb{u}=\{1,-1\}$ with weights $\vb{a}=\{1/(2h),-1/(2h)\}$. At location $x_1$ the standard stencil is used: $\vb{v}_1=\vb{u}$ and $\vb{b}_1=\vb{a}$. At location $x_{-1}$ a forward-approximation is used: $\vb{v}_{-1}=\{1,0,-1\}$ and $\vb{b}_{-1}=\{-1/(2h), 2/h, -3/(2h)\}$. Composition using these results in a series of $(\{-1/2, 0, -1/24,\ldots\},\beta=2)$, which is clearly not an approximation to the second-derivative. Other combinations may result in the negative of the expected result, zero, or something else completely. Note that it is perfectly acceptable to mix approximations at different target locations. For example, at a domain boundary we can use all one-sided approximations for the inner stencil and center-approximations at the interior.

With this in mind, our results demonstrate that it is possible to construct complex differential stencils with guaranteed accuracy via the composition of lower-derivative approximations. This is the first step towards facilitating the formation of large-scale linear systems of arbitrary partial differential equations in a systematic and automatic manner. Future work will include the implementation of these concepts into numerical tools for the wider community, and an investigation of the composition between interpolation and differentiation operations, similar to those in the Closest Point Method.

\bibliography{refs}

\begin{thebibliography}{27}
\providecommand{\natexlab}[1]{#1}
\providecommand{\url}[1]{\texttt{#1}}
\expandafter\ifx\csname urlstyle\endcsname\relax
  \providecommand{\doi}[1]{doi: #1}\else
  \providecommand{\doi}{doi: \begingroup \urlstyle{rm}\Url}\fi

\bibitem[Stone(1990)]{stone}
HA~Stone.
\newblock A simple derivation of the time-dependent convective-diffusion
  equation for surfactant transport along a deforming interface.
\newblock \emph{Physics of Fluids A: Fluid Dynamics}, 2\penalty0 (1):\penalty0
  111--112, 1990.

\bibitem[Olsen et~al.(1998)Olsen, Maini, and Sherratt]{olsen}
Luke Olsen, Philip~K Maini, and Jonathan~A Sherratt.
\newblock Spatially varying equilibria of mechanical models: Application to
  dermal wound contraction.
\newblock \emph{Mathematical biosciences}, 147\penalty0 (1):\penalty0 113--129,
  1998.

\bibitem[Tian et~al.(2009)Tian, Macdonald, and Ruuth]{tian}
Li~Tian, Colin~B Macdonald, and Steven~J Ruuth.
\newblock Segmentation on surfaces with the closest point method.
\newblock In \emph{2009 16th IEEE International Conference on Image Processing
  (ICIP)}, pages 3009--3012. IEEE, 2009.

\bibitem[Arad et~al.(1997)Arad, Yakhot, and Ben-Dor]{arad}
M~Arad, A~Yakhot, and G~Ben-Dor.
\newblock A highly accurate numerical solution of a biharmonic equation.
\newblock \emph{Numerical Methods for Partial Differential Equations: An
  International Journal}, 13\penalty0 (4):\penalty0 375--391, 1997.

\bibitem[Appel{\"o} and Petersson(2009)]{fd1}
Daniel Appel{\"o} and N~Anders Petersson.
\newblock A stable finite difference method for the elastic wave equation on
  complex geometries with free surfaces.
\newblock \emph{Communications in Computational Physics}, 5\penalty0
  (1):\penalty0 84--107, 2009.

\bibitem[Duru and Virta(2014)]{fd2}
Kenneth Duru and Kristoffer Virta.
\newblock Stable and high order accurate difference methods for the elastic
  wave equation in discontinuous media.
\newblock \emph{Journal of Computational Physics}, 279:\penalty0 37--62, 2014.

\bibitem[Wang et~al.(2016)Wang, Virta, and Kreiss]{fd3}
Siyang Wang, Kristoffer Virta, and Gunilla Kreiss.
\newblock High order finite difference methods for the wave equation with
  non-conforming grid interfaces.
\newblock \emph{Journal of Scientific Computing}, 68\penalty0 (3):\penalty0
  1002--1028, 2016.

\bibitem[Burman and Hansbo(2012)]{fe1}
Erik Burman and Peter Hansbo.
\newblock Fictitious domain finite element methods using cut elements: Ii. a
  stabilized nitsche method.
\newblock \emph{Applied Numerical Mathematics}, 62\penalty0 (4):\penalty0
  328--341, 2012.

\bibitem[Hansbo and Hansbo(2002)]{fe2}
Anita Hansbo and Peter Hansbo.
\newblock An unfitted finite element method, based on nitsche’s method, for
  elliptic interface problems.
\newblock \emph{Computer methods in applied mechanics and engineering},
  191\penalty0 (47-48):\penalty0 5537--5552, 2002.

\bibitem[Hansbo et~al.(2014)Hansbo, Larson, and Zahedi]{fe3}
Peter Hansbo, Mats~G Larson, and Sara Zahedi.
\newblock A cut finite element method for a stokes interface problem.
\newblock \emph{Applied Numerical Mathematics}, 85:\penalty0 90--114, 2014.

\bibitem[Demird{\v{z}}i{\'c} and Muzaferija(1994)]{fv1}
Ismet Demird{\v{z}}i{\'c} and Samir Muzaferija.
\newblock Finite volume method for stress analysis in complex domains.
\newblock \emph{International journal for numerical methods in engineering},
  37\penalty0 (21):\penalty0 3751--3766, 1994.

\bibitem[Gong et~al.(2013)Gong, Xuan, Ming, and Zhang]{fv2}
Jingfeng Gong, Lingkuan Xuan, Pingjian Ming, and Wenping Zhang.
\newblock An unstructured finite-volume method for transient heat conduction
  analysis of multilayer functionally graded materials with mixed grids.
\newblock \emph{Numerical Heat Transfer, Part B: Fundamentals}, 63\penalty0
  (3):\penalty0 222--247, 2013.

\bibitem[Tang et~al.(2005)Tang, Qiu, Zhang, and Yang]{tang}
Ping Tang, Feng Qiu, Hongdong Zhang, and Yuliang Yang.
\newblock Phase separation patterns for diblock copolymers on spherical
  surfaces: A finite volume method.
\newblock \emph{Physical Review E}, 72\penalty0 (1):\penalty0 016710, 2005.

\bibitem[Macdonald and Ruuth(2009)]{cp1}
Colin~B. Macdonald and Steven~J. Ruuth.
\newblock The implicit {C}losest {P}oint {M}ethod for the numerical solution of
  partial differential equations on surfaces.
\newblock \emph{SIAM J. Sci. Comput.}, 31\penalty0 (6):\penalty0 4330--4350,
  2009.

\bibitem[Chen and Macdonald(2015)]{cp2}
Yujia Chen and Colin~B. Macdonald.
\newblock The {C}losest {P}oint {M}ethod and multigrid solvers for elliptic
  equations on surfaces.
\newblock \emph{SIAM J. Sci. Comput.}, 37\penalty0 (1), 2015.

\bibitem[Macdonald and Ruuth(2008)]{ls1}
Colin~B. Macdonald and Steven~J. Ruuth.
\newblock Level set equations on surfaces via the {C}losest {P}oint {M}ethod.
\newblock \emph{J. Sci. Comput.}, 35\penalty0 (2--3):\penalty0 219--240, 2008.

\bibitem[Elliott(2008)]{ls2}
Charles Elliott.
\newblock An eulerian level set method for partial differential equations on
  evolving surfaces.
\newblock \emph{Comput. Vis. Sci.}, 13:\penalty0 17--22, 2008.

\bibitem[Macdonald(2008)]{macthesis}
Colin~B. Macdonald.
\newblock \emph{The closest point method for time-dependent processes on
  surfaces}.
\newblock PhD thesis, Simon Fraser University, Dept. of Mathematics, 2008.

\bibitem[Selvadurai(2000)]{APS}
A.~P.~S. Selvadurai.
\newblock \emph{Partial Differential Equations in Mechanics 2: The Biharmonic
  Equation, Poisson’s Equation}.
\newblock Springer, Berlin, 2000.

\bibitem[Ford(2014)]{ford}
William Ford.
\newblock \emph{Numerical linear algebra with applications: Using MATLAB}.
\newblock Academic Press, 2014.

\bibitem[Jokisaari et~al.(2017)Jokisaari, Voorhees, Guyer, Warren, and
  Heinonen]{spinodal}
Andrea~M Jokisaari, PW~Voorhees, Jonathan~E Guyer, James Warren, and
  OG~Heinonen.
\newblock Benchmark problems for numerical implementations of phase field
  models.
\newblock \emph{Computational Materials Science}, 126:\penalty0 139--151, 2017.

\bibitem[Cahn(1961)]{cahn}
John~W Cahn.
\newblock On spinodal decomposition.
\newblock \emph{Acta metallurgica}, 9\penalty0 (9):\penalty0 795--801, 1961.

\bibitem[Balay et~al.(2022)Balay, Abhyankar, Adams, Benson, Brown, Brune,
  Buschelman, Constantinescu, Dalcin, Dener, Eijkhout, Gropp, Hapla, Isaac,
  Jolivet, Karpeev, Kaushik, Knepley, Kong, Kruger, May, McInnes, Mills,
  Mitchell, Munson, Roman, Rupp, Sanan, Sarich, Smith, Zampini, Zhang, Zhang,
  and Zhang]{petsc}
Satish Balay, Shrirang Abhyankar, Mark~F. Adams, Steven Benson, Jed Brown,
  Peter Brune, Kris Buschelman, Emil~M. Constantinescu, Lisandro Dalcin, Alp
  Dener, Victor Eijkhout, William~D. Gropp, V\'{a}clav Hapla, Tobin Isaac,
  Pierre Jolivet, Dmitry Karpeev, Dinesh Kaushik, Matthew~G. Knepley, Fande
  Kong, Scott Kruger, Dave~A. May, Lois~Curfman McInnes, Richard~Tran Mills,
  Lawrence Mitchell, Todd Munson, Jose~E. Roman, Karl Rupp, Patrick Sanan,
  Jason Sarich, Barry~F. Smith, Stefano Zampini, Hong Zhang, Hong Zhang, and
  Junchao Zhang.
\newblock {PETS}c {W}eb page.
\newblock \url{https://petsc.org/}, 2022.
\newblock URL \url{https://petsc.org/}.

\bibitem[Mishra(2022)]{mishra}
Abhishek Mishra.
\newblock \emph{Enabling Computational Methods for Discretization of Partial
  Differential Equation Models using Stencil Composition}.
\newblock PhD thesis, State University of New York at Buffalo, 2022.

\bibitem[Chang et~al.(2018{\natexlab{a}})Chang, Fabien, Knepley, and
  Mills]{tas}
Justin Chang, Maurice~S Fabien, Matthew~G Knepley, and Richard~T Mills.
\newblock Comparative study of finite element methods using the
  time-accuracy-size (tas) spectrum analysis.
\newblock \emph{SIAM Journal on Scientific Computing}, 40\penalty0
  (6):\penalty0 C779--C802, 2018{\natexlab{a}}.

\bibitem[Homolya and Ham(2016)]{weakscaling}
Mikl{\'o}s Homolya and David~A Ham.
\newblock A parallel edge orientation algorithm for quadrilateral meshes.
\newblock \emph{SIAM Journal on Scientific Computing}, 38\penalty0
  (5):\penalty0 S48--S61, 2016.

\bibitem[Chang et~al.(2018{\natexlab{b}})Chang, Nakshatrala, Knepley, and
  Johnsson]{staticscaling}
Justin Chang, KB~Nakshatrala, Matthew~G Knepley, and L~Johnsson.
\newblock A performance spectrum for parallel computational frameworks that
  solve pdes.
\newblock \emph{Concurrency and Computation: Practice and Experience},
  30\penalty0 (11):\penalty0 e4401, 2018{\natexlab{b}}.

\end{thebibliography}

\newpage
\appendix

\section{\normalsize{Derivation of first-order accurate third derivative stencil using stencil composition}}\label{sec:appendix}

If stencil $A$ corresponds to $f'(x_0)$,
\begin{align*}
    f'(x_0) &\longrightarrow \frac{f(x_1) - f(x_0)}{h} + \mathcal{O}(h),
\end{align*}
and, stencil $B$ corresponds to $f''(x_0)$,
\begin{align*}
    f''(x_0) &\longrightarrow \frac{f(x_2) - 2f(x_1)+f(x_0)}{h^2} + \mathcal{O}(h).
\end{align*}

This results in $\vb{u}=\{1,0\}$ and $\vb{v}=\{2,1,0\}$ as the associated integer vectors with weights $\vb{a}=\{1/h,-1/h\}$ and $\vb{b}=\{1/h^2,-2/h^2,1/h^2\}$, respectively.

The composition $B(A)$ can thus be derived in the following way,

\begin{align*}
    B(A) &= \sum_j \sum_i a_i b_j  \left(\sum_{|\alpha| \ge 0} h^{\alpha} \dfrac{\left(\vb{u}_i+\vb{v}_j\right)^{\alpha}}{\alpha!} f^{(\alpha)}(\vb{x}_0)\right)\nonumber\\
        &=\sum_j b_j
            \left[
                \dfrac{1}{h}\left(\sum_{|\alpha| \ge 0} h^{\alpha} \dfrac{\left(1+\vb{v}_j\right)^{\alpha}}{\alpha!} f^{(\alpha)}(\vb{x}_0)\right)
                -\dfrac{1}{h}\left(\sum_{|\alpha| \ge 0} h^{\alpha} \dfrac{\vb{v}_j^{\alpha}}{\alpha!} f^{(\alpha)}(\vb{x}_0)\right)
            \right]\nonumber\\
        &=\dfrac{1}{h} \left[
                \dfrac{1}{h^2}\left(\sum_{|\alpha| \ge 0} h^{\alpha} \dfrac{3^{\alpha}}{\alpha!} f^{(\alpha)}(\vb{x}_0)\right)
                -\dfrac{2}{h^2}\left(\sum_{|\alpha| \ge 0} h^{\alpha} \dfrac{2^{\alpha}}{\alpha!} f^{(\alpha)}(\vb{x}_0)\right)
                +\dfrac{1}{h^2}\left(\sum_{|\alpha| \ge 0} h^{\alpha} \dfrac{1^{\alpha}}{\alpha!} f^{(\alpha)}(\vb{x}_0)\right)
            \right]\nonumber\\
        & \qquad-\dfrac{1}{h} \left[
                \dfrac{1}{h^2}\left(\sum_{|\alpha| \ge 0} h^{\alpha} \dfrac{2^{\alpha}}{\alpha!} f^{(\alpha)}(\vb{x}_0)\right)
                -\dfrac{2}{h^2}\left(\sum_{|\alpha| \ge 0} h^{\alpha} \dfrac{1^{\alpha}}{\alpha!} f^{(\alpha)}(\vb{x}_0)\right)
                +\dfrac{1}{h^2}\left(\sum_{|\alpha| \ge 0} h^{\alpha} \dfrac{0^{\alpha}}{\alpha!} f^{(\alpha)}(\vb{x}_0)\right)
            \right]\nonumber\\
        &=\dfrac{1}{h^3}\sum_{|\alpha| \ge 0} h^{\alpha} \dfrac{3^{\alpha}-3(2)^{\alpha}+3}{\alpha!} f^{(\alpha)}(\vb{x}_0)-\dfrac{1}{h^3}f(\vb{x}_0)\nonumber\\
        &\longrightarrow \dfrac{1}{h^3}(\{0,0,0,1,\frac{3}{2},\frac{5}{4},\ldots\},\beta=0)=(\{1,\frac{3}{2},\frac{5}{4},\ldots\},\beta=3).
\end{align*}

\section{Calculation of stability requirement for $\partial_t f=-\partial_x\left(\partial_x\left(\partial_{xx}f\right)\right)$}\label{sec:appendixB}

Consider the discretization of $\partial_t f=-\partial_x\left(\partial_x\left(\partial_{xx}f\right)\right)$ about a point $x_0$. Composing the inner two derivatives using
standard center-finite differences results in
\begin{align*}
    \dfrac{\partial}{\partial x}\left(\dfrac{\partial f}{\partial xx}\right) = \dfrac{-f_{-2}+2f_{-1}-2f_{+1}+f_{+2}}{2h^3},
\end{align*}
where $f_{i}=f(x_0+i h)$ and is the standard, compact second-order accurate discretization for $\partial_{xxx}f$. Composing this with the first-derivative results in 
\begin{align*}
    \dfrac{\partial}{\partial x}\left(\dfrac{\partial f}{\partial xxx}\right) = \dfrac{f_{-3}-2f_{-2}-f_{-1}+4f_{0}-f_{+1}-2f_{+2}+f_{+3}}{4h^4},
\end{align*}
which has wider support compared to the compact schemes which uses points $[-2,2]$.

As before assume that the solution for the $n^{th}$ time step is $f_j^n=\xi_k^n e^{\imath k j h}$ where $k$ is the wave mode and $\xi_k$ is the growth factor. Using this in the discretization, dividing by $f_0^n$ results, and solving for $\xi_k$ results in
\begin{equation*}
    \xi_k = 1 + \dfrac{\Delta t}{2h^4}\left(\cos{hk}+2\cos{2hk}-\cos{3hk}-2\right).
\end{equation*}

Stability requires that $|\xi_k|\leq 1$ and thus
\begin{align*}
    -1 & \leq 1 + \dfrac{\Delta t}{2h^4}\left(\cos{hk}+2\cos{2hk}-\cos{3hk}-2\right) \leq 1\\
    -2 & \leq \dfrac{\Delta t}{2h^4}\left(\cos{hk}+2\cos{2hk}-\cos{3hk}-2\right) \leq 0 \nonumber\\
    -4 & \leq \dfrac{\Delta t}{h^4}\left(\cos{hk}+2\cos{2hk}-\cos{3hk}-2\right) \leq 0.
\end{align*}
As $\left(\cos{hk}+2\cos{2hk}-\cos{3hk}-2\right)\in [-128/27,0]$ we have
\begin{equation}
    0 \leq \dfrac{\Delta t}{h^4} \leq \dfrac{27}{32},
\end{equation}
which results in a time-step restriction of $\Delta t\leq 27 h^4/32$.

Both the compact finite difference approximation and composing $\partial_{xx}\left(\partial_{xx}\right)$ give a discretization of
\begin{align*}
    \dfrac{\partial f}{\partial xxxx} = \dfrac{f_{-2}-4f_{-1}+6f_{0}-4f_{+1}+f_{+2}}{h^4},
\end{align*}
which results in a growth factor of 
\begin{equation*}
    \xi_k = 1+\dfrac{\Delta t}{h^4}\left(8\cos{hk}-2\cos{2hk}-6\right).
\end{equation*}
Ensuring that $|\xi_k|\leq 1$ requires that $\Delta t\leq h^4/8$.

\end{document}